\documentclass[9pt,reqno]{amsart}  

\usepackage[english]{babel}
\usepackage[p,osf]{cochineal}
\usepackage[scale=.95,type1]{cabin}
\usepackage[zerostyle=c,scaled=.94]{newtxtt}
\usepackage[T1]{fontenc} 
\usepackage[utf8]{inputenc} 
\usepackage{amsthm}
\usepackage{amsmath}
\usepackage{amssymb}
\usepackage{amsfonts}
\usepackage{amsaddr}
\usepackage{xspace,subcaption}
\usepackage{enumitem} 
\usepackage{csquotes}
\usepackage{xcolor}
\usepackage[colorlinks]{hyperref}
\hypersetup{colorlinks, breaklinks, citecolor=teal,filecolor=black}
\usepackage{graphicx}
\usepackage{mathtools}
\mathtoolsset{centercolon}
\usepackage{bm}
\usepackage{pgfkeys}
\usepackage{pgf,interval}
\intervalconfig{soft open fences}

\DeclareMathOperator{\Tr}{Tr}

\DeclareMathOperator{\NCP}{NC}
\DeclareMathOperator{\NCG}{NCG}
\DeclareMathOperator{\scf}{sc}

\DeclareMathOperator{\E}{\mathbf{E}}
\DeclareMathOperator{\Prob}{\mathbf{P}}

\DeclareMathOperator{\pTr}{pTr}

\newcommand{\ii}{\mathrm{i}}

\newcommand{\F}{\mathbf{F}}

\ifdefined\C
\renewcommand{\C}{\mathbf{C}}
\else
\newcommand{\C}{\mathbf{C}}
\fi

\newcommand{\un}{\underline}
\newcommand{\vx}{\bm{x}}

\newcommand{\vy}{\bm{y}}

\newcommand{\wt}{\widetilde}

\newcommand{\R}{\mathbf{R}}
\newcommand{\N}{\mathbf{N}}
\newcommand{\Z}{\mathbf{Z}}

\newcommand{\DD}{\mathbf{D}}

\newcommand{\cO}{\mathcal{O}}
\newcommand{\co}{{\scriptstyle\mathcal{O}}}

\newcommand{\cP}{\pi}
\newcommand{\cQ}{\sigma}

\newcommand{\dif}{\operatorname{d}\!{}}
\DeclarePairedDelimiter{\braket}{\langle}{\rangle}%
\DeclarePairedDelimiter{\abs}{\lvert}{\rvert}%
\DeclarePairedDelimiter{\norm}{\lVert}{\rVert}%
\providecommand\given{}
\newcommand\SetSymbol[1][]{\nonscript\:#1\vert\allowbreak\nonscript\:\mathopen{}}
\DeclarePairedDelimiterX{\tuple}[1](){\renewcommand\given{\SetSymbol[\delimsize]}#1}
\DeclarePairedDelimiterX{\set}[1]\{\}{\renewcommand\given{\SetSymbol[\delimsize]}#1}
\DeclarePairedDelimiterXPP{\landauO}[1]{\cO}(){}{#1}
\DeclarePairedDelimiterXPP{\landauo}[1]{\co}(){}{#1}
\DeclarePairedDelimiterXPP{\landauOprec}[1]{\cO_\prec}(){}{#1}

\usepackage[giveninits=true,url=false,doi=false,isbn=false,eprint=true,datamodel=mrnumber,sorting=nty,maxcitenames=4,maxbibnames=99,backref=false,block=space,backend=biber,bibstyle=phys]{biblatex} 
\AtEveryBibitem{\clearfield{month}}
\AtEveryCitekey{\clearfield{month}} 
\renewbibmacro{in:}{}
\DeclareFieldFormat[article]{title}{\emph{#1}} 
\DeclareFieldFormat{mrnumber}{\ifhyperref{\href{http://www.ams.org/mathscinet-getitem?mr=#1}{\nolinkurl{MR#1}}}{\nolinkurl{#1}}}
\DeclareFieldFormat{pmid}{\ifhyperref{\href{https://www.ncbi.nlm.nih.gov/pubmed/#1}{\nolinkurl{PMID#1}}}{\nolinkurl{#1}}}
\DeclareFieldFormat{eprint}{\ifhyperref{\href{https://arxiv.org/abs/#1}{\nolinkurl{arXiv:#1}}}{\nolinkurl{#1}}}
\renewbibmacro*{doi+eprint+url}{%
  \iftoggle{bbx:doi}{\printfield{doi}}{}
  \newunit\newblock%
  \printfield{mrnumber}%
  \newunit\newblock%
  \printfield{pmid}%
  \newunit\newblock%
  \printfield{eprint}%
  \iftoggle{bbx:url}{\usebibmacro{url+urldate}}{}}
\usepackage{ifdraft}
\usepackage{tikz,pgfplots}
\usepackage{etoolbox,listings,xstring,xparse,xstring}
\usepackage{xargs,xcolor}

\usetikzlibrary{graphs,positioning,calc,graphs.standard,shapes}

\definecolor{cactusgreen}{HTML}{75A644}

\tikzset{
foreachCode/.code={\xdef\Lst{}\foreach #1 },
br/.style = {bend right},
b0/.style = {bend left=0},
bl/.style = {bend left=40},
gr/.style = {draw=lightgray},
w2/.style = {fill=white,draw=white,ultra thick,rounded corners},
w3/.style = {fill=black,draw=black,ultra thick,rounded corners},
g2/.style = {fill=gray,draw=gray,ultra thick,rounded corners},
g3/.style = {fill=gray,draw=gray,ultra thick,rounded corners}
}

\makeatletter
\define@key{circEmb}{radius}{\def\cgr@radius{#1}}
\define@key{circEmb}{ptwidth}{\def\cgr@ptwidth{#1}}
\define@key{circEmb}{dual}{\def\cgr@dual{#1}}
\define@key{circEmb}{bgcol}{\def\cgr@bgcol{#1}}
\define@key{circEmb}{number}{\def\cgr@number{#1}}
\define@key{circEmb}{parts}{\def\cgr@parts{#1}}
\define@key{circEmb}{part}{\def\cgr@part{#1}}
\define@key{circEmb}{eparts}{\def\cgr@eparts{#1}}
\define@key{circEmb}{epart}{\def\cgr@epart{#1}}
\define@key{circEmb}{edges}{\def\cgr@edges{#1}}
\define@key{circEmb}{cactedges}{\def\cgr@cactedges{#1}}
\define@key{circEmb}{vlabel}{\def\cgr@vlabel{#1}}
\define@key{circEmb}{diamond}{\def\cgr@diamond{#1}}
\define@key{circEmb}{vertlabels}{\def\cgr@vertlabels{#1}}
\setkeys{circEmb}{radius=1em,ptwidth=4pt,dual=0,number=4,part=,parts={},epart=,eparts={},edges={},vlabel=0,cactedges={},bgcol=lightgray,diamond=0,vertlabels={}}%
\newcommand{\circEmb}[2][]{%
\begingroup%
\setkeys{circEmb}{#1}%
\StrCount{\cgr@part}{,}[\partLength]
\StrCount{\cgr@vertlabels}{,}[\vertlabelsLength]
\StrCount{\cgr@epart}{,}[\epartLength]
\ifnum\vertlabelsLength=0 
  \def\cgr@vertlabels {1,...,\cgr@number}
\else 
  \pgfmathsetmacro\cgr@number {1+\vertlabelsLength}
\fi
\def\degree {360/\cgr@number}
\ifnum\cgr@diamond=1 \def\diamondshape{diamond} \else \def\diamondshape{circle} \fi
\begin{tikzpicture}[baseline={([yshift=-2pt]current bounding box.center)},font=\footnotesize]
  \pgfdeclarelayer{bg} 
  \pgfsetlayers{bg,main}
  \begin{pgfonlayer}{bg}
    \draw[radius=\cgr@radius,fill=\cgr@bgcol, draw=none] (0,0) circle ;
  \end{pgfonlayer}
  \foreach \v [count=\s] in \cgr@vertlabels { 
  \ifnum\s=\cgr@number \def \shape {\diamondshape} \else \def \shape {circle} \fi
  \node [draw=black, \shape , minimum width=\cgr@ptwidth, minimum height=\cgr@ptwidth, inner sep=0pt,fill=black, text=white] at ({\degree * (\s -.5) + 90}:\cgr@radius) (\v) {\ifnum\cgr@vlabel=1 $\v$\fi };
  \coordinate (\v-c) at (\v.center);
  }
  \ifnum\cgr@dual=1
  \foreach \v [count=\s] in \cgr@vertlabels { 
  \ifnum\s=\cgr@number \def \shape {\diamondshape} \else \def \shape {circle} \fi 
  \node [text=white,draw=black, \shape , minimum width=\cgr@ptwidth, minimum height=\cgr@ptwidth, inner sep=0pt,fill=gray] at ({\degree * (\s) + 90}:\cgr@radius) (e \v) {\ifnum\cgr@vlabel=1 $\v$\fi };
  \coordinate (e-\v-c) at (e \v.center);
  }
  \fi
  \begin{pgfonlayer}{bg}
    \foreach \set in \cgr@parts {
    \draw [w3,foreachCode={\a in \set {\xdef\Lst{\Lst (\a-c) -- }}}] \Lst cycle;
    }
    \ifnum\partLength>0
    \draw [w3,foreachCode={\a in \cgr@part {\xdef\Lst{\Lst (\a-c) -- }}}] \Lst cycle;
    \fi
    \foreach \set in \cgr@eparts {
    \draw [g3,foreachCode={\a in \set {\xdef\Lst{\Lst (e-\a-c) -- }}}] \Lst cycle;
    }
    \ifnum\epartLength>0
    \draw [g3,foreachCode={\a in \cgr@epart {\xdef\Lst{\Lst (e-\a-c) -- }}}] \Lst cycle;
    \fi
  \end{pgfonlayer}
  \foreach \a/\b in \cgr@cactedges { \draw[draw=lightgray,line width=3pt] (\a)--(\b); }
  \foreach \a/\b in \cgr@edges { \draw[draw=white] (\a)--(\b); }
  #2
\end{tikzpicture}%
\endgroup%
}
\makeatother

\bibliography{refs} 

\numberwithin{equation}{section}  
\newtheorem{theorem}{Theorem}[section]
\newtheorem{assumption}[theorem]{Assumption}
\newtheorem{lemma}[theorem]{Lemma}

\newtheorem{definition}[theorem]{Definition}
\newtheorem{example}[theorem]{Example}

\newtheorem{remark}[theorem]{Remark}

\newtheorem{corollary}[theorem]{Corollary}

\newtheorem{extension}[theorem]{Extension}

\date{\today}
\author{Giorgio Cipolloni \and L\'aszl\'o Erd\H{o}s}
\address{IST Austria, Am Campus 1, 3400 Klosterneuburg, Austria}
\author{Dominik Schr\"oder\(^{\dagger}\)}
\address{Institute for Theoretical Studies, ETH Zurich, Clausiusstr.\ 47, 8092 Zurich, Switzerland}
\email{giorgio.cipolloni@ist.ac.at} 
\email{lerdos@ist.ac.at}
\email{dschroeder@ethz.ch}
\thanks{\(^{\dagger}\)Supported by Dr.\ Max R\"ossler, the Walter Haefner Foundation and the ETH Z\"urich Foundation}
\subjclass[2010]{60B20, 15B52, 46L54} 
\keywords{Global Law, Local Law, Non-crossing Partitions, Asymptotic Freeness}
\title{Thermalisation for Wigner matrices}
\date{\today}
\pgfplotsset{compat=1.17}
\begin{document}

\begin{abstract} We compute the deterministic approximation of  products of Sobolev
    functions of large Wigner matrices \(W\) and provide  an optimal error bound on their fluctuation
    with very high probability. This generalizes Voiculescu's seminal theorem~\cite{MR1094052} from
     polynomials to general Sobolev functions, as well as from tracial quantities to individual matrix
     elements. Applying the result to  $e^{\ii t W}$ for large \(t\), 
       we obtain a precise decay  rate  for the overlaps of several 
    deterministic matrices with temporally well separated Heisenberg time evolutions; thus 
    we demonstrate the thermalisation effect of the unitary group generated by Wigner matrices.
    \end{abstract}

\thispagestyle{empty}    


\maketitle

\section{Introduction}

Since E. Wigner's  pioneering idea~\cite{MR77805}, random matrices are ubiquitously used
 to model complex quantum Hamiltonians. Most work deal with  the spectacular universality phenomenon of local eigenvalue statistics~\cite{MR0220494} but the applicability of  random matrix theory goes well beyond. The current paper has been motivated to understand the joint distribution of  the unitary operator \(e^{\ii t W}\), i.e.\ the quantum evolution  corresponding to a large \(N\times N\) Wigner matrix \(W\), at different, typically large times.
  
More generally, in our main result we compute the leading deterministic approximation
for the random quantity
\begin{equation}\label{fa}
\braket{f_1(W) A_1 f_2(W) A_2\ldots f_k(W)A_k},
\end{equation}
and we provide an optimal error bound on its fluctuation. Here \(f_i\)'s are Sobolev test-functions,
\(A_i\)'s are bounded deterministic matrices (observables) and
\(\braket{R} :=\frac{1}{N}\Tr R\) denotes the normalized trace of any matrix \(R\in\C^{N\times N}\).
 The deterministic approximation is a sum of several explicit terms, labelled by non-crossing partitions of \(k\) elements. Whenever all \(f_i=p_i\) are polynomials,
  such formulas are routinely generated in free probability theory 
  by evaluating \(\tau( p_1(s) a_1 p_2(s) a_2\ldots p_k(s)a_k)\)
  in a non-commutative \(*\)-algebra  \(\mathcal{A}\) with a tracial state \(\tau\), 
   where \(s\) is a semicircular
  element  and the set \(\{a_1, a_2, \ldots, a_k\}\) is \emph{freely independent} of \(s\).   Voiculescu's classical result~\cite{MR1094052} and its extensions from Gaussian (GUE) to general  Wigner matrices and to include deterministic matrices, see~\cite[Theorem~5.4.5]{MR2760897} and~\cite[Sect.~4,Thm.~20]{MR3585560}, assert that
 \begin{equation}\label{fa1}
\E \langle p_1(W) A_1 p_2(W) A_2\ldots p_k(W)A_k\rangle \to \tau( p_1(s) a_1 p_2(s) a_2\ldots p_k(s)a_k),
\end{equation}
where the \(k\)-tuple  \((a_1, \ldots, a_k)\in \mathcal{A}^k\) is the
distributional  limit of \((A_1, \ldots, A_k)\in (\C^{N\times N})^k\) as \(N\to\infty\).
Several independent Wigner matrices can also be considered on the left hand side; they are modelled
by freely independent semicircular elements in the right hand side.

 Our Theorem~\ref{theo:intgs} extends~\eqref{fa1} in several important directions. First, we can handle general Sobolev functions \(f_i\in H^2(\R)\) and not only polynomials  since we circumvent  the moment method used
 in free probability theory. We can even consider certain \(N\)-dependent functions living on mesoscopic scales.
 Second, we control the convergence in~\eqref{fa1} immediately in very high probability and not only in expectation,
 saving additional variance and high moment calculations typically performed separately with the moment method.
 This strengthening allows us to directly handle several independent random matrices instead of a single \(W\),
 just by simple conditioning; the similar extension in  the standard free probability approach requires considerably more
 sophisticated combinatorics. Third, we  obtain an optimal error term of order \(N^{-1}\) involving the \(k\)-th Sobolev norms
 of \(f_i\)  and we have a freedom to trade in weaker bounds for less smoothness assumption down to \(f_i\in H^2\). Fourth, we obtain similar deterministic approximations with optimal error terms not only for the normalized traces~\eqref{fa} but for all matrix elements \(\braket{\vx, f_1(W) A_1 \ldots f_k(W)A_k \vy}\) with 
 any deterministic vectors \(\vx,\vy\in \C^N\). Note that  individual matrix
 elements have no counterpart in the limiting algebra \(\mathcal{A}\), so they are beyond the scope of
 standard free probability theory. Finally, our deterministic approximations are obtained before the \(N\to \infty\) limit is taken, hence the convergence of the deterministic matrices \(A_1, \ldots, A_k\) is not required.

In our main applications  we  consider~\eqref{fa} with the exponential functions \(f_j(x) = e^{\ii s_j x}\) 
  and we are primarily interested in the decay of~\eqref{fa} for large times \(s_j\gg 1\).
 This problem has two related motivations originating from mathematical
physics and free probability theory, respectively, that we briefly explain.

The classical RAGE theorem~\cite[Section~5.4]{MR883643} for self-adjoint operators  \(H\) on an infinite dimensional Hilbert space \(\mathcal{H}\)  shows that   the Heisenberg 
time evolution \(A(t)= e^{\ii tH}A e^{-\ii tH}\) of  a compact operator \(A\)  asymptotically vanishes on 
any state \(\psi\in \mathcal{H}\) in the continuous spectral subspace of \(H\); more precisely  \(\langle \psi, A(t) \psi\rangle\)
tends to zero in Cesaro mean for large time \(t\). 
Since acting on a finite dimensional 
space, large \(N\times N\) Wigner matrices \(W\) 
 do not have continuous spectrum in a literal sense, but for many physical purposes they still behave 
as an operator with continuous spectrum; for example their eigenvectors are completely delocalized~\cite{MR2481753,MR2871147,2007.09585}. Hence the analogue of the RAGE theorem for 
Wigner matrices would assert that the matrix elements of  \(A(t):=e^{\ii tW}A e^{-\ii tW}\) at
any fixed deterministic  vectors \(\vx,\vy\in \C^N\) become
very close to their limiting value for large times,  i.e.
\begin{equation}\label{uAv}
\braket{\vx, A(t) \vy}\approx \braket{\vx,\vy}\braket{A}  \qquad \text{for}\quad t\gg 1.
\end{equation}
We  call this phenomenon \emph{thermalisation} as it corresponds to a decay to a certain equilibrium.
Similarly, for  two bounded deterministic matrices (observables) \(A\) and \(B\)  one  expects
that \(A(t)\) and \(B\) become thermalised, i.e. 
\begin{equation}\label{AB}
    \braket{ A(t) B}\approx \braket{ A}\braket{ B} \qquad \text{for}\quad t\gg 1.
 \end{equation}

 Exact equalities are not expected in~\eqref{uAv} and~\eqref{AB} even after the \(t\to \infty\) limit
 as a 
  consequence of the finite dimensionality.  Our Theorem~\ref{theo:intgs} in this context proves the thermalisation mechanism
  with a precise decay rate for large times, in particular we show  that
 \begin{equation}\label{AB1}
    \begin{split}
      \braket{\vx,A(t)\vy} &= \braket{\vx,\vy} \braket{A}+ \theta(t)^2\frac{\braket{\vx,\mathring A\vy}}{t^3} + \landauO*{N^\epsilon\frac{t^2}{N^{1/2}}} \\
      \braket{A(t)B} &= \braket{A}\braket{B} + \theta(t)^2\frac{\braket{\mathring A\mathring B}}{t^3} + \landauO*{N^\epsilon\frac{t^2}{N}}
    \end{split}
 \end{equation}
holds with very high probability for the oscillatory order one function \(\theta(t):=J_1(2t)t^{1/2}\), with $J_1$ a Bessel function of the first kind, and where \(\mathring A:=A-\braket{A}\) denotes the traceless part of \(A\). We thus obtain an \emph{approximate RAGE theorem} for Wigner matrices with a precise decay rate in time and with an \(N\)-dependent error bound due to the finite dimensionality of the system.  The effective error terms in~\eqref{AB1} allow for a simultaneous limit for large  \(N\) and  \(t\) in a certain range.  Interestingly, the inverse cubic decay rate stems from the square root singularity 
 of the Wigner semicircle density at the spectral edges. Since this square root behaviour is typical for
 the density of states in a large class of random matrix ensembles~\cite{MR3684307,MR4164728},
 the cubic decay rate is expected to be fairly universal. For brevity, in this paper we focus on the simplest Wigner
 case, deferring the more general ensembles to future work.
 
 We obtain similar thermalisation results for the multiple time evolutions of several  observables
 and identify the precise rate of time decay in each case.  
 The deterministic approximation 
has a hierarchical structure that allows us to identify the sector with 
the slowest (dominant) thermalisation rate.  We find that if some observables  or their products 
are traceless, the thermalisation is enhanced.  
 For example, if 
 \(B(s)= e^{\ii sW} B e^{-\ii sW}\) is the time evolution of another deterministic \(B\) matrix with \(s\gg1\), \(t-s\gg 1\) and \(C\) is a third observable, then we obtain
 \begin{equation}\label{ABC}
  \begin{split}
    \braket{ A(t) B(s) C} &= \braket{A}\braket{B}\braket{C} + \theta(s)^2\frac{\braket{A}\braket{\mathring B\mathring C}}{s^3} + \theta(t)^2\frac{\braket{B}\braket{\mathring A\mathring C}}{t^3} + \theta(t-s)^2\frac{\braket{C}\braket{\mathring A\mathring B}}{(t-s)^3}\\
    &\quad +\theta(s)\theta(t)\theta(t-s)\frac{\braket{\mathring A\mathring B \mathring C}}{s^{3/2}t^{3/2}(t-s)^{3/2}}+\landauO*{N^\epsilon \frac{t^3}{N}}
  \end{split}
 \end{equation}
 with very high probability. Note that the prevailing decay rate is strongly influenced by the possible vanishing of some of the numerators in~\eqref{ABC}. In particular,    if all three observables
  are traceless, \(\langle A\rangle =\langle B\rangle=\langle C\rangle =0\), and the large times \(t\), \(s\) and \(t-s\) are comparable, the decay rate is the \(\frac{9}{2}\)-th power of the time.

%
%

%

 Our second motivation 
 comes from  Voiculescu's  theorem~\cite[Theorem~5.4.5]{MR2760897} and~\cite[Sect.~4,Thm.~20]{MR3585560} (see also~\cite{MR1094052,MR1207936,MR1601878} for previous results)
 which asserts
  that independent  \(N\times N\) Wigner matrices, \(W_1, W_2, \ldots W_k\),
 are asymptotically free. This means that for any collection of polynomials \(p_1, p_2, \ldots, p_r\)
 that are (asymptotically) traceless, i.e.\ \(\braket{ p_j(W)} \to 0\) as \(N\to\infty\), we have 
 \begin{equation}\label{free}
   \braket{ p_1(W_{i_1}) p_2(W_{i_2}) \ldots p_r(W_{i_r}) } \to 0, \qquad \text{as} \quad N\to\infty,
 \end{equation}
  in expectation and almost surely, where the product is \emph{alternating} in the sense that \(i_1\ne i_2\), \(i_2\ne i_3, \ldots, i_{r-1}\ne i_r\). 
 
 The asymptotic freeness property~\eqref{free} of \emph{independent} Wigner matrices is a fundamental 
result that connects random matrices with free probability.  Using the thermalisation mechanism we
show that  not only independent Wigner matrices are asymptotically free, but different
long time evolutions by the \emph{very same} Wigner matrix  also make deterministic observables asymptotically free. 
More precisely, we show that the Heisenberg time evolutions of arbitrary deterministic observables, \(A_1(t_1), A_2(t_2),
\ldots, A_k(t_k)\) are asymptotically free whenever all time differences \(|t_i-t_j|\) are very large. Equivalently, we prove
 that with very high probability for any polynomials \(p_1,\ldots,p_r\)
\begin{equation}\label{AAA}
  \braket{ p_1(A_{i_1}(t_{i_1}))\cdots p_r(A_{i_r}(t_{i_r}))} \to 0 \quad \text{as}\quad N\to\infty\quad \text{and}\quad \min_{i\ne j}|t_i-t_j|\to \infty,
\end{equation}
whenever \(i_1\ne i_2\),  \(i_2\ne i_3, \ldots, i_{r-1}\ne i_r\) and \(\braket{p_j(A_{i_j}(t_{i_j}))}\to 0\) for all \(j=1,\ldots, k\). The precise statement with effective error bounds is given in  Corollary~\ref{corr free indep}. 
We stress that
the mechanism  to obtain asymptotic freeness via thermalisation in
~\eqref{AAA} is very different from the one behind~\eqref{free} relying on  independence. A freeness mechanism similar to ours was demonstrated for different powers of the \emph{same} Haar unitary matrix by Haagerup and Larsen in~\cite[Lemma~3.7]{MR1784419}.

 In order to understand~\eqref{fa},  we first derive a new \emph{multi-resolvent local law} 
 in Theorem~\ref{local law}, i.e.\ we identify the deterministic approximation of 
  \( G(z_1) A_1 G(z_2) A_2\ldots G(z_k)\)
 for the resolvents, \(G(z)= (W-z)^{-1}\), and then extend it to general Sobolev functions
 via the Helffer-Sj\"ostrand calculus. For a single resolvent the deterministic approximation \(G(z)\approx m(z) I\) is given by the unique scalar solution \(m=m(z)\) to the Dyson equation \(-1/m=m+z\), both in \emph{averaged sense}, \(\braket{ G(z)} \approx m(z)\), and in \emph{isotropic sense}, 
  \(\braket{\vx,  G(z)\vy} \approx m(z)\braket{\vx,\vy}\)
  for any vectors \(\vx,\vy\in \C^N\). 
   The  multi-resolvent local law is proven by recursively
 analysing a  system of self-consistent  equations that is  an adapted version of the deterministic Schwinger-Dyson equation obtained from second moment  Gaussian  calculation.
 The  fluctuation term in this approximation has been  estimated in our recent work~\cite{2012.13215}.
 Our approach also works in the mesoscopic regime, i.e.\ when the imaginary 
 part of the spectral parameter in \(G(z)\) is small as  a negative power of \(N\).
 In turn, this allows us to analyse the unitary  time evolution  \(e^{\ii tW}\) for very long, even \(N\)-dependent,
 times. The mesoscopic regime, however, requires to identify a multiple cancellation
 effect in the deterministic approximation. Amusingly, we need two very different, but eventually 
 equivalent formulas for this approximation; 
 one is based on non-crossing graphs (see Lemma~\ref{lemma q} later) and arises naturally from the recursive structure of the Dyson equation. The other one from~\eqref{M 1..k alt} is a partial resummation of the first one in terms of non-crossing partitions and the free cumulant function of divided differences of \(m\), manifesting the cancellation.

Voiculescu's theorem~\eqref{fa1} or 
asymptotic freeness in the form~\eqref{free} for independent Wigner matrices 
has traditionally been proven with the moment method using  very involved combinatorics.
It efficiently handles polynomials of fixed degree as stated in~\eqref{free} and can be extended
to general functions by polynomial approximation. However,  to obtain effective controls  (e.g  explicit speed of convergence)
or possibly $N$ dependent test functions (like mesoscopic linear statistics)
usually  requires high (\(N\)-dependent) degree for the polynomials that,
 in turn,
are increasingly difficult for the moment method as well as for the analytic subordination method~\cite{MR3717087}. 
Thus the extension of the  moment method 
to more general functions has natural  limitations, although there is a remarkable recent development 
for rational functions~\cite{MR3718048,MR3967396,MR3730345,2103.05962}. 
The trace of a smooth cut-off function of a polynomial in GUE and deterministic matrices 
has been analysed via the Master equation and linearization
  in~\cite{MR2183281,MR3000553} for the purpose of identifying the
norm of the polynomial. Recently general smooth functions  were considered in the same setup with a new interpolation method between the GUE matrices and their infinite dimensional limits, the semicircular elements~\cite{1912.04588}. A large \(N\)-expansion to  arbitrary order was also obtained~\cite{2011.04146}. We follow a  different route 
via the local laws for resolvents that works for general Wigner matrices and also for matrix elements, it handles mesoscopic regimes  very efficiently and it yields optimal control in very high probability sense offering an alternative
to the customary free probability approach.

\subsection*{Notation and conventions}
We introduce some notations we use throughout the paper. For integers \(l,k\in\N \) we use the notations \([k]:= \set{1,\ldots, k}\), and
\[ 
[k,l):=\{ k, k+1,\ldots, l-1\}, \qquad [k,l]:=\{ k, k+1,\ldots, l-1, l\}
\]
for \(k< l\). For positive quantities \(f,g\) we write \(f\lesssim g\) and \(f\sim g\) if \(f \le C g\) or \(c g\le f\le Cg\), respectively, for some constants \(c,C>0\) which depend only on the constants appearing in~\eqref{eq:momentass}. We denote vectors by bold-faced lower case Roman letters \({\bm x}, {\bm y}\in\C ^N\), for some \(N\in\N\). Vector and matrix norms, \(\norm{\vx}\) and \(\norm{A}\), indicate the usual Euclidean norm and the corresponding induced matrix norm. For any \(N\times N\) matrix \(A\) we use the notation \(\braket{ A}:= N^{-1}\Tr  A\) to denote the normalized trace of \(A\). Moreover, for vectors \({\bm x}, {\bm y}\in\C^N\) we define
\[ \braket{ {\bm x},{\bm y}}:= \sum \overline{x}_i y_i, \qquad A_{\vx\vy}:=\braket{\vx,A\vy},\]
with \(A\in\C^{N\times N}\). For any $z\in\mathbf{C}$, by $\Re z$ and $\Im z$ we denote the real and imaginary part of $z$, respectively. We will use the concept of ``with very high probability'' meaning that for any fixed \(D>0\) the probability of the \(N\)-dependent event is bigger than \(1-N^{-D}\) if \(N\ge N_0(D)\). Moreover, we use the convention that \(\xi>0\) denotes an arbitrary small constant which is independent of \(N\).

\subsubsection*{Acknowledgement} 
The authors are very grateful to Roland Speicher for useful correspondence on the problem and pointing out additional references. The authors also thank the anonymous referees whose comments significantly improved the readability of the manuscript, and also Jana Reker for carefully reading the manuscript and spotting several typos.

\section{Main results}
We consider real symmetric or complex Hermitian \(N\times N\) Wigner matrices \(W\). We formulate the following assumptions on the entries of \(W\). 
\begin{assumption}\label{ass:entr}
  We assume that the matrix elements \(w_{ab}\) are independent up to the Hermitian symmetry \(w_{ab}=\overline{w_{ba}}\) and identically distributed in the sense that \(w_{ab}\stackrel{\mathrm{d}}{=} N^{-1/2}\chi_{\mathrm{od}}\), for \(a<b\), \(w_{aa}\stackrel{\mathrm{d}}{=}N^{-1/2} \chi_{\mathrm{d}}\), with \(\chi_{\mathrm{od}}\) being a real or complex random variable and \(\chi_{\mathrm{d}}\) being a real random variable such that \(\E \chi_{\mathrm{od}}=\E \chi_{\mathrm{d}}=0\) and \(\E |\chi_{\mathrm{od}}|^2=1\). In the complex case we also assume that \(\E \chi_{\mathrm{od}}^2\in\R\). In addition, we assume the existence of the high moments of \(\chi_{\mathrm{od}}\), \(\chi_{\mathrm{d}}\), i.e.\ that there exist constants \(C_p>0\), for any \(p\in\N \), such that
  \begin{equation}\label{eq:momentass}
    \E \abs*{\chi_{\mathrm{d}}}^p+\E \abs*{\chi_{\mathrm{od}}}^p\le C_p.
  \end{equation}
\end{assumption}
Our main result is the asymptotic evaluation of products of multiple time-evolved observables \(e^{-\ii t W}Ae^{\ii t W}\) for general deterministic matrices \(A\). More generally, we prove that alternating products of functions of Wigner and deterministic matrices like \(\braket{f(W)Ag(W)B\dots}\) with high probability concentrate around a deterministic limit which we compute explicitly. In order to state the result we first introduce \emph{non-crossing partitions}~\cite{MR309747} and related objects.  
\begin{definition}[Lattice of non-crossing partitions]
  Let \(S\subset\N\) be a finite set of integers. We call a partition \(\cP\) of the set \(S\) \emph{crossing} if there exist blocks \(B\ne B'\in\cP\) with \(a,b\in B\), \(c,d\in B'\) and \(a<c<b<d\), otherwise we call it \emph{non-crossing} and we denote the set of non-crossing partitions by \(\NCP(S)\). For each non-crossing partition \(\cP=\set{B_1,\ldots,B_n}\in\NCP(S)\) we denote the number of blocks in the partition by \(\abs{\cP}:=n\).

  We define a partial order \(\le\) on \(\NCP(S)\), the \emph{refinement order}, such that \(\cP\le \cQ\) if and only if \(\cP\) is a refinement of \(\cQ\), i.e.\ if for each \(B\in\cP\) there exists \(B'\in\cQ\) such that \(B\subset B'\). The partially ordered set \((\NCP(S),\le)\) is in fact a lattice as any two \(\cP,\cQ\in\NCP(S)\) admit unique least upper and greatest lower bounds \(\cP\vee\cQ,\cP\wedge\cQ\in\NCP(S)\). Moreover, there exist unique maximal and minimal elements \(0_S,1_S\in\NCP(S)\) defined by \(0_S:=\set{\set{a}\given a\in S}\) and \(1_S:=\set{S}\). 
\end{definition} 
The following definition is combinatorially identical to the definition of free cumulants of random variables in free probability in terms of the trace functional (see~\cite[Section 4]{MR1268597} or~\cite{MR3374523} for connections with classical, Boolean and monotone cumulants). 
\begin{definition}[Free cumulant function]\label{m kappa def}
  Fix \(k\in\N\) and let \(f\colon 2^{[k]}\to\C\) be a function mapping subsets of \([k]\) to scalars. We then implicitly define the \emph{free cumulant function of \(f\)} as the unique map \(f_\circ\colon 2^{[k]}\to \C\) satisfying that for any \(S\subset[k]\) we have 
  \begin{equation}\label{eq m kappa def}
    f[S] = \sum_{\cP\in\NCP(S)} \prod_{B\in \cP} f_\circ[B].
  \end{equation}
\end{definition}
The implicit relation~\eqref{eq m kappa def} in Definition~\ref{m kappa def} can be recursively turned into an explicit definition of \(f_\circ\). Indeed, for \(\abs{S}=1\) the relation~\eqref{eq m kappa def} implies \(f_\circ[i]=f[i]\), so that using~\eqref{eq m kappa def} for \(\abs{S}=2\) it follows that \(f_\circ[i,j]=f[i,j]-f[i]f[j]\). For general \(S\subset [k]\) the free cumulant function can be written explicitly as 
\begin{equation}\label{f circ expl}
  f_\circ[S]=\sum_{\cP\in\NCP(S)} \mu(\cP,1_S)\prod_{B\in \cP} f[B], \quad \mu(\cP,\cQ):= \begin{cases}
    1, & \cP=\cQ,\\ -\sum_{\pi<\nu\le \sigma} \mu(\nu,\sigma),& \pi<\sigma,
  \end{cases}
\end{equation}
in terms of the \emph{M\"obius function} \(\mu\colon\set{(\pi,\sigma)\given \pi\le \sigma\in\NCP(S)}\to \Z\), see Lemma~\ref{mobius expl} later for an alternative non-recursive definition. We note that for \(\abs{S}>3\) the M\"obius function depends on the elements of the blocks and not only 
on the block sizes of \(\cP\), e.g.\ \(\mu(\set{12|3|4},\set{1234})=2\ne 1=\mu(\set{13|2|4},\set{1234})\). This is because the concept of non-crossing partition relies on the ordering of \(\N\); the M\"obius function for the lattice of \emph{all} partitions would be a function of the 
block sizes alone.

Non-crossing partitions have an alternative geometrical definition. Arrange the elements of \(S\) equidistantly in counter-clockwise order on the circle and for each \(B\in\cP\) consider the \(P_B\) convex hull of the points \(x\in B\). Then \(\cP\) is non-crossing if and only if the polygons \(\set{P_B\given B\in\cP}\) are pairwise disjoint, see Figure~\ref{example partitions} for some examples (note that the partition in Figure~\eqref{ex19b} is a refinement of the partition in Figure~\ref{ex19a}). We note that for any \(\cP\in\NCP(S)\) the complement \(\DD\setminus\cup_{B\in\cP}P_B\) of the polygons \(P_B,B\in\cP\) in the disk \(\DD\) has \(\abs{S}-\abs{\cP} +1\) connected components. The geometrical interpretation is particularly useful for defining the Kreweras complement of non-crossing partitions~\cite{MR309747}.

\begin{figure}[htbp]
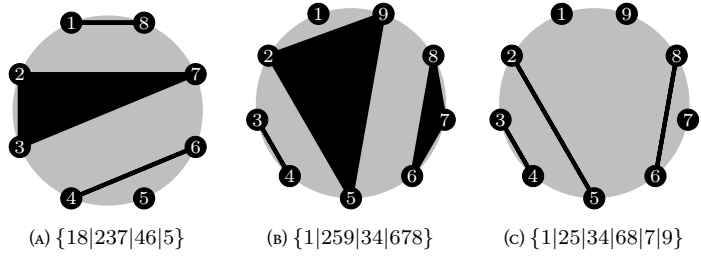

\centering
\subcaptionbox{\(\set{18|237|46|5}\)\label{ex18}}[.25\linewidth]{
  \circEmb[ptwidth=8pt,number=8,parts={{2,3,7},{4,6},{1,8}},vlabel=1,radius=4em]{}}
\subcaptionbox{\(\set{1|259|34|678}\)\label{ex19a}}[.25\linewidth]{
  \circEmb[ptwidth=8pt,number=9,parts={{2,5,9},{3,4},{6,7,8}},vlabel=1,radius=4em]{}}
\subcaptionbox{\(\set{1|25|34|68|7|9}\)\label{ex19b}}[.25\linewidth]{
  \circEmb[ptwidth=8pt,number=9,parts={{2,5},{3,4},{6,8}},vlabel=1,radius=4em]{}}
\caption{Some partitions together with the corresponding disjoint convex hulls}\label{example partitions}
\end{figure} 

\begin{definition}[Kreweras complement]\label{kreweras}
  Let \(S\subset\N\) be a set of integers equidistantly arranged in counter-clockwise order on the circle and label the arcs between the points also by \(S\) in such a way that the arc \(x\) succeeds the point \(x\) in counter-clockwise order. Then for \(\cP\in\NCP(S)\) we define the \emph{Kreweras complement} \(K(\cP)\in\NCP(S)\) such that \(x,y\in S\) belong to the same block of \(K(\pi)\) if and only if the arcs \(x,y\) are in the same connected component of \(\DD\setminus\cup_{B\in\cP}P_B\).
\end{definition}
In Figure~\ref{ex dual part} we give two examples of partitions and their Kreweras complements. We note that \(\abs{\pi}+\abs{K(\pi)}=\abs{S}+1\) for any \(\cP\in\NCP(S)\). Moreover, \(K^2=K\circ K\) is simply a rotation in the sense that \(K^2(\cP)\) is the partition where for \(S=\set{s_1,\ldots,s_n}\) the elements in each block of \(\cP\) are shifted by \(s_1\mapsto s_2\mapsto \cdots \mapsto s_n\mapsto s_1\). In particular the map \(K\) on \(\NCP(S)\) is invertible. 
\begin{figure}[htbp]
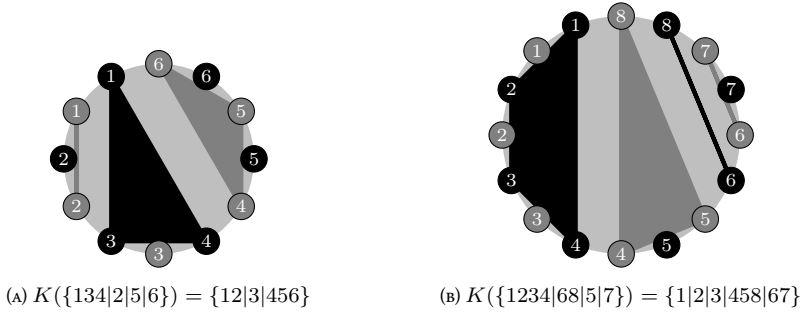

\centering
\subcaptionbox{\(K(\set{134|2|5|6})=\set{12|3|456}\)\label{ex21}}[.48\linewidth]{
    \circEmb[radius=4em,number=6,dual=1,vlabel=1,ptwidth=10pt,part={1,3,4},eparts={{1,2},{4,5,6}}]{}} 
\subcaptionbox{\(K(\set{1234|68|5|7})=\set{1|2|3|458|67}\)\label{ex22}}[.48\linewidth]{\circEmb[radius=5em,number=8,dual=1,vlabel=1,ptwidth=10pt,parts={{1,2,3,4},{6,8}},eparts={{4,5,8},{6,7}}]{}}
\caption{Example of partitions (in black) and their Kreweras complements (in dark gray). 
}\label{ex dual part}
\end{figure}


We are now ready to state our main result. We define the semicircular average of any function \(f\colon[-2,2]\to\C\) as 
\[\braket{f}_\mathrm{sc}:=\int_{-2}^2 f(x)\rho_\mathrm{sc}(x)\dif x,\quad \rho_\mathrm{sc}(x):=\frac{\sqrt{4-x^2}}{2\pi}.\]
Furthermore, we define the \(\cP\)-\emph{partial}-\emph{trace} and the \(\cP\)-\emph{trace} for partitions \(\cP\in\NCP[k]\).
\begin{definition}\label{def partial trace}Let \((\mathcal{A},\braket{\cdot})\) be a complex tracial algebra. Then for \(k\in\N\) and \(\cP\in\NCP[k]\) we define the \emph{\(\cP\)-partial-trace}\footnote{An analogous partial trace is used in free probability theory for the representation of operator valued conditional expectations, c.f.\ for example the formula for $\widetilde \phi_\sigma$ above \cite[Theorem~19]{MR3585560}.} \(\pTr_\cP\colon \mathcal{A}^{k-1}\to\mathcal{A}\) as 
  \begin{equation}\label{pi trace}
    \begin{split}
      \braket{A_1,\ldots,A_k}_\pi &:= \prod_{B\in\pi}\braket[\Big]{\prod_{j\in B}A_j},\\
      \pTr_\cP(A_1,\ldots,A_{k-1}) &:=  \biggl( \prod_{j\in B(k)\setminus\set{k}} A_j\biggr) \prod_{\substack{B\in\cP\\B\not\ni k}} \braket*{\prod_{j\in B} A_j}, 
    \end{split}
  \end{equation}
  where \(B(k)\in\cP\) denotes the block containing \(k\), and all products are ordered increasingly in the indices. 
\end{definition}
\begin{theorem}\label{theo:intgs}
  Let \(k\ge 2\), let \(A_1,\ldots, A_k\) be deterministic matrices with \(\norm{A_i}\lesssim 1\), and let \(f_1,\ldots,f_k\) be Sobolev functions \(f_i\in H^k([-3,3])\) normalised such that \(\norm{f_i}_{L^\infty}\sim 1\). Then for any \(\xi>0\) and any deterministic vectors \(\vx,\vy\) with \(\norm{\vx}+\norm{\vy}\lesssim 1\) we have  
  \begin{align}\nonumber
      \braket{f_1(W)A_1 \ldots f_k(W)A_k } &= \sum_{\pi\in\NCP[k]}\braket{A_1,\ldots,A_k}_{K(\pi)}\prod_{B\in \pi} \scf_\circ[B]   +\landauO*{ N^\xi\frac{\max_i\norm{f_i}_{H^k}}{N}}\\\label{eq:niceintform}
      \braket{\vx,f_1(W)A_1 \ldots f_k(W)\vy } &= \sum_{\pi\in\NCP[k]}\braket{\vx,\pTr_{K(\pi)}(A_1,\ldots,A_{k-1})\vy} \prod_{B\in \pi} \scf_\circ[B]  \\\nonumber
      &\qquad + \landauO*{ N^\xi\frac{\max_i\norm{f_i}_{H^k}}{N^{1/2}}}
  \end{align}
  with very high probability,  where \(\scf_\circ\) is the free cumulant function from Definition~\ref{m kappa def} of \(\scf[i_1,\ldots,i_n]:=\braket{f_{i_1}f_{i_2}\cdots f_{i_n}}_\mathrm{sc}\). For \(k=1\) the same result holds with \(f\in H^k\) and \(\norm{\cdot}_{H^k}\) replaced by \(f\in H^2\) and \(\norm{\cdot}_{H^2}\), respectively. A straightforward generalization of~\eqref{eq:niceintform} to include several independent Wigner 
  matrices is given in Extension~\ref{rmk:manyW}. 
\end{theorem}

Note that by eigenvalue rigidity (see e.g.~\cite[Theorem 7.6]{MR3068390} or~\cite{MR2871147}) the spectrum of \(W\) is contained in \([-2-\epsilon,2+\epsilon]\), for any small \(\epsilon>0\), with very high probability, hence \(f(W)\) is well defined for \(f\in H^k([-3,3])\),
i.e. functions defined only on $[-3,3]$. In fact, without loss of generality, we may assume that \(f\in H_0^k([-3,3])\)
by multiplying the original $f$ with a smooth cutoff function without changing $f(W)$. 
  (see Section~\ref{sec:mresnew} for more details).

\begin{remark}  The average version of~\eqref{eq:niceintform} 
for polynomial test functions \(f_i\) has a long history.  In expectation sense the first result of this
type was proved  for several  independent Gaussian 
(GUE) random matrices with \(A_i=I\) in  Voiculescu's 
seminal paper~\cite{MR1094052};  later upgraded  to 
almost sure convergence in~\cite{MR1601878}. The extension to Wigner matrices with
general entry distribution  as well as the inclusion of
special block diagonal deterministic matrices  \(A_i\) was achieved in~\cite{MR1207936}.
The  case with arbitrary deterministic matrices  can be found in~\cite[Theorem~5.4.5]{MR2760897} and~\cite[Sect.~4,Thm.~20]{MR3585560} (see also~\cite{MR3711884} under relaxed moment conditions
 on the entry distribution). The only results beyond polynomials are in~\cite{MR3730345, 2103.05962} for certain class of rational functions; general Sobolev functions have not been considered before the current work. Furthermore, the isotropic version of~\eqref{eq:niceintform} is new even for polynomials.

 We note that in the language of free probability theory the r.h.s.\ of~\eqref{eq:niceintform} can be interpreted as follows. If the deterministic matrices \(A_i\) converge in the sense of moments to some elements \(a_i\in\mathcal{A}\), 
 \[A_i\overset{\mathrm{distr}}{\longrightarrow}a_i,\quad \text{as}\quad  N\to\infty,\] 
 of some non-commutative probability space \((\mathcal{A},\phi)\), then we asymptotically have
 \[\braket{f_1(W)A_1\cdots f_k(W)A_k}\to \phi(f_1(s)a_1\cdots f_k(s)a_k)\]
 for a semicircular element \(s\in\mathcal A\) freely independent of \(a_1,\ldots,a_k\).
\end{remark}
%

\begin{example}\label{example theo}
For \(k=1,2,3\) the deterministic approximation \(F_k\) of \(f_1(W)A_1\cdots f_k(W)\) in~\eqref{eq:niceintform} is given as follows. The deterministic approximations of \(\braket{f_1(W)A_1\cdots f_k(W)A_k}\) follow by multiplying the expressions below by \(A_k\) and taking the trace. 
\begin{itemize}
\item[\((k=1)\)] Here we simply have \(F_1=\braket{f_1}_\mathrm{sc}\) since \(\scf_\circ[1]=\braket{f_1}_\mathrm{sc}\).
\item[\((k=2)\)] We have
\[
\begin{split}
  F_2&=A_1\scf_\circ[1]\scf_\circ[2]+\braket{A_1}\scf_\circ[1,2]\\
  &=A_1\braket{f_1}_\mathrm{sc}\braket{f_2}_\mathrm{sc}+\braket{A_1}(\braket{f_1f_2}_\mathrm{sc}-\braket{f_1}_\mathrm{sc}\braket{f_2}_\mathrm{sc})
\end{split}\]
using
\(\scf_\circ[1,2]=\braket{f_1f_2}_\mathrm{sc}-\braket{f_1}_\mathrm{sc}\braket{f_2}_\mathrm{sc}\).
\item[\((k=3)\)] For \(k=3\) there are five terms (corresponding to the non-crossing partitions visualised in Figure~\ref{ex3})
\[\begin{split}
  F_3&= A_1A_2 \scf_\circ[1]\scf_\circ[2]\scf_\circ[3]+ 
  \braket{A_1A_2}\scf_\circ[1,3]\scf_\circ[2] + A_1 \braket{A_2}\scf_\circ[1]\scf_\circ[2,3] \\
  &\quad + A_2\braket{A_1}\scf_\circ[1,2]\scf_\circ[3] + \braket{A_1}\braket{A_2}\scf_\circ[1,2,3]
\end{split}\] 
with 
\[
\begin{split}
  \scf_\circ[1,2,3] = \braket{f_1f_2f_3}_\mathrm{sc}-\braket{f_1f_2}_\mathrm{sc}\braket{f_3}_\mathrm{sc}-\braket{f_1f_3}_\mathrm{sc}\braket{f_2}_\mathrm{sc}-\braket{f_2f_3}_\mathrm{sc}\braket{f_1}_\mathrm{sc}+2\braket{f_1}_\mathrm{sc}\braket{f_2}_\mathrm{sc}\braket{f_3}_\mathrm{sc}
\end{split}\]
and \(\scf_\circ[i],\scf_\circ[i,j]\) as before. 
\end{itemize}
\end{example}
\begin{figure}[htbp]
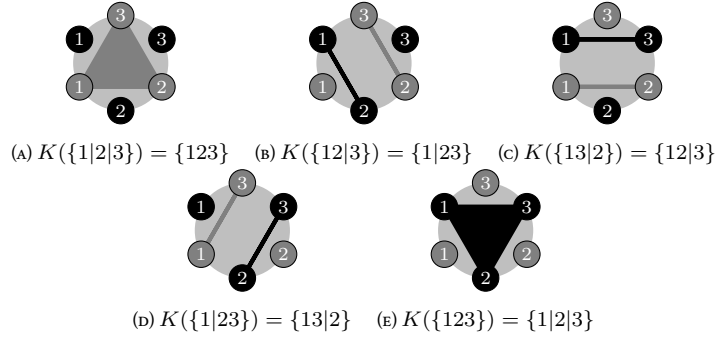

  \centering
  \subcaptionbox{\(K(\set{1|2|3})=\set{123}\)\label{ex31}}[.25\linewidth]{
    \circEmb[radius=2em,number=3,dual=1,ptwidth=10pt,epart={1,2,3},vlabel=1]{}}
  \subcaptionbox{\(K(\set{12|3})=\set{1|23}\)\label{ex32}}[.25\linewidth]{
    \circEmb[radius=2em,number=3,dual=1,ptwidth=10pt,epart={2,3},part={1,2},vlabel=1]{}}
  \subcaptionbox{\(K(\set{13|2})=\set{12|3}\)\label{ex33}}[.25\linewidth]{
    \circEmb[radius=2em,number=3,dual=1,ptwidth=10pt,epart={1,2},part={1,3},vlabel=1]{}}
  \subcaptionbox{\(K(\set{1|23})=\set{13|2}\)\label{ex34}}[.25\linewidth]{
    \circEmb[radius=2em,number=3,dual=1,ptwidth=10pt,epart={1,3},part={2,3},vlabel=1]{}}
  \subcaptionbox{\(K(\set{123})=\set{1|2|3}\)\label{ex35}}[.25\linewidth]{
    \circEmb[radius=2em,number=3,dual=1,ptwidth=10pt,part={1,2,3},vlabel=1]{}}
  \caption{List of all non-crossing partitions \(\NCP([3])\) on three vertices (in black) together with their Kreweras complement (in dark gray).}\label{ex3}
\end{figure}

\subsection{Thermalisation and asymptotic freeness}
We now specialise Theorem~\ref{theo:intgs} to the functions \(f(x):=e^{\ii s x}\), with \(s\in\mathbf{R}\), and define
\begin{equation}\label{eq:defphi}
  \varphi(s):=\braket{e^{\ii s \cdot}}_\mathrm{sc}=\int_{-2}^2\rho_\mathrm{sc}(x)e^{\ii s x}\dif x=\frac{J_1(2s)}{s},
\end{equation} 
where \(J_1\) is a Bessel function of the first kind. 
We note that by standard asymptotic of the Bessel function we have
\begin{equation}\label{eq:largeasj}
   J_1(x)=-\cos\Bigl(x+\frac{\pi}{4}\Bigr)\sqrt{\frac{2}{\pi x}}+\landauO*{\frac{1}{x^{3/2}}}, \qquad \text{for} \quad x\gg 1.
\end{equation}

\begin{corollary}\label{corr:intgs}
Let \(k\ge 2\), \(s_1,\ldots,s_k\in\R\) and let \(A_1,\ldots, A_k,\vx,\vy\) be deterministic matrices and vectors with \(\norm{A_i}\lesssim 1\) and \(\norm{\vx}+\norm{\vy}\lesssim1\). Then 
\begin{align}\nonumber
\braket{e^{\ii s_1 W}A_1\cdots A_{k-1}e^{\ii s_k W}A_k}&=\sum_{\pi\in \NCP[k]}\braket{A_1,\ldots,A_k}_{K(\pi)}\prod_{B\in \pi} \varphi_\circ [B] +\landauO*{ N^\xi \frac{\max_i\abs{s_i}^k}{N}},\\
\label{eq:niceintform corr}
\braket{\vx, e^{\ii s_1 W}A_1\cdots A_{k-1}e^{\ii s_k W}\vy}&=\sum_{\pi\in \NCP[k]}\braket{{\bm x},\pTr_{K(\pi)}(A_1,\ldots,A_{k-1}){\bm y}}\prod_{B\in \pi} \varphi_\circ [B] \\\nonumber
&\qquad +\landauO*{ N^\xi \frac{\max_i\abs{s_i}^k}{N^{1/2}}},
\end{align}
with very high probability for any \(\xi>0\), where \(\varphi_\circ\) is the free cumulant function from Definition~\ref{m kappa def} of \(\varphi[i_1,\ldots,i_n]:=\varphi(s_{i_1}+\cdots+s_{i_n})\), with \(\varphi\) being defined in~\eqref{eq:defphi}. 
\end{corollary}
Corollary~\ref{corr:intgs} ensures a time decay if some or all of the involved matrices are traceless. More precisely, we obtain: 
\begin{corollary}
\label{cor:misscor}
  Let \(k\ge 2\), and let \(s_1,\ldots,s_k\in\R\). Then 
  \begin{enumerate}[label=(\roman*)]
    \item for deterministic traceless matrices \(A_1,\ldots, A_k\) we have the averaged estimate
    \begin{equation}\label{eq therm s} 
      \abs*{\braket*{e^{\ii s_1 W} A_1\cdots e^{\ii s_k W} A_k }} \lesssim  \max_{i\ne j}\Bigl(\frac{1}{(1+\abs{s_i})(1+\abs{s_j})}\Bigr)^{3/2}+ N^\xi \frac{\max_i \abs{s_i}^{k} }{N},
    \end{equation}
    \item and for deterministic traceless matrices \(A_1,\ldots, A_{k-1}\) we have the isotropic estimate
    \begin{equation}\label{eq therm s iso} 
      \abs*{\braket*{\vx,e^{\ii s_1 W} A_1\cdots A_{k-1}e^{\ii s_k W} \vy }} \lesssim  \max_{i}\Bigl(\frac{1}{1+\abs{s_i}}\Bigr)^{3/2}+ N^\xi \frac{\max_i \abs{s_i}^k }{N^{1/2}},
    \end{equation}
  \end{enumerate}
  both with very high probability.
\end{corollary}
In particular, for the unitary time evolution 
    \begin{equation}
      A(t):=e^{\ii tW}A e^{-\ii t W}
    \end{equation} 
    and times \(t_1,\ldots,t_k\in\R\) with consecutive differences \(s_i:=t_i-t_{i-1}\), \(t_0:=t_k\) we have the very high probability bound
    \begin{equation}\label{eq therm} 
      \abs{\braket{A_1(t_1)\ldots A_k(t_k)}} \lesssim  \max_{i\ne j}\Bigl(\frac{1}{(1+\abs{s_i})(1+\abs{s_j})}\Bigr)^{3/2} + N^\xi \frac{\max_i \abs{s_i}^{k} }{N},
    \end{equation}
  as a consequence of~\eqref{eq therm s} and similarly for the isotropic case. 
   
  In the first term of~\eqref{eq therm}  we observe a thermalisation decay for $k$ traceless observables $A_i$, each of them
  evolved some time $t_i$, as long as at least two  consecutive time differences grow.
   If only some of the observables are traceless 
  but the times $t_i$ are  \emph{ordered} (equivalently, all but one $s_i$ in Corollary~\ref{cor:misscor}
have the same sign), then we have a
   decay factor for each traceless observable, hence the decay rate is typically much faster. This is the content
  of the following corollary.
  
\begin{corollary}[Thermalisation decay]\label{cor therm}
  Let \(k\ge 2\), and let \(t_1<\cdots<t_k\) be ordered times. Then 
  \begin{enumerate}[label=(\roman*)]
    \item if \(\mathfrak{a}\) of the deterministic matrices \(A_1,\ldots,A_k\) are traceless, then we have the averaged bound
    \begin{equation}\label{eq therm a}
      \abs{\braket{A_1(t_1)\ldots A_k(t_k)}} \lesssim  \max_{i}\Bigl(\frac{1}{1+\abs{s_i}}\Bigr)^{3(1+\lceil \mathfrak{a}/2\rceil)/2} +N^\xi \frac{\max_i \abs{s_i}^{k} }{N},
    \end{equation}
    for \(s_i=t_i-t_{i-1}\), \(t_0=t_k\),
    \item and if \(\mathfrak{a}\) of the deterministic matrices \(A_1,\ldots,A_{k-1}\) are traceless, then we have the isotropic bound 
    \begin{equation}\label{eq therm a iso}
      \abs{\braket{\vx, A_1(t_1)\ldots A_k(t_k)\vy}} \lesssim  \max_{i}\Bigl(\frac{1}{1+\abs{s_i}}\Bigr)^{3 \lceil \mathfrak{a}/2\rceil/2}+N^\xi \frac{\max_i \abs{s_i}^{k} }{N^{1/2}} ,
    \end{equation}
    for \(s_1=t_1,s_k=t_k\) and \(s_i=t_i-t_{i-1}\) otherwise.
  \end{enumerate}
\end{corollary}
The estimate~\eqref{eq therm} in particular implies that time evolutions \(A(t_1),A(t_2)\) become \emph{asymptotically free} as \(\abs{t_1-t_2}\to\infty\). We recall that elements \(a_1,\ldots,a_p\) of some non-commutative probability space \((\mathcal A,\braket{\cdot})\) are called \emph{free} if for any \(k\in\N\), any \(i_1\ne i_2\ne \cdots\ne i_k\in[p]\) and polynomials \(p_1,\ldots,p_k\) satisfying \(\braket{p_j(a_{i_j})}=0\) it holds that 
\begin{equation}
  \braket{p_1(a_{i_1})\cdots p_k(a_{i_k})}=0.
\end{equation}
\begin{corollary}\label{corr free indep}
  Let \(p\in\N\) and let \(A_1,\ldots,A_p\) be sequences of deterministic matrices with \(\norm{A_i}\lesssim1\). Consider $p$
  sequences of times  $\{ t_1^N\}_{N\in\mathbf{N}}$,  $\{ t_2^N\}_{N\in\mathbf{N}}\ldots,$  $\{ t_p^N\}_{N\in\mathbf{N}}$
   such that $\min_{i\ne j}|t_i^N-t_j^N|\to \infty$ as $N\to \infty$, then the unitary time evolutions \(A_1(t_1),\ldots,A_p(t_p)\) are asymptotically free. More precisely, for any \(k\in\N\) and \(i_1\ne \cdots \ne i_k\in[p]\), and any polynomials \(p_1,\ldots,p_k\) with \(\braket{p_j(A_{i_j})}=0\) (which also implies $\braket{p_j(A_{i_j}(t_{i_j}))}=0)$ we have  (with $t_i=t_i^N$ for brevity)
  \begin{equation}\label{eq corr free}
\lim_{N\to\infty \atop \min_{i\ne j}\abs{t_i-t_j}\to\infty} \braket{p_1(A_{i_1}(t_{i_1}))\cdots p_k(A_{i_k}(t_{i_k}))}
  \end{equation}
  with very high probability, as long as \(\max_i\abs{t_i}\lesssim N^{1/k-\epsilon}\) for some \(\epsilon>0\).
\end{corollary}
Corollary~\ref{corr free indep} can be easily extended to the asymptotic freeness of \(A_1(t_1),\ldots,A_p(t_p)\) and the algebra generated by arbitrary deterministic matrices \(\set{D_1,\ldots,D_q}\).



 
We stress that conjugation by the unitary time evolutions with respect to the very same Wigner matrix yields asymptotically free observables in~\eqref{eq corr free}. In the language of free-probability theory, a statement analogous to Corollary~\ref{corr free indep} is\footnote{We thank Roland Speicher for bringing this to our attention in private communication, and pointing out the reference~\cite[Proposici\'on~3.21]{jacobo} for the simple proof of the assertion.} that for self-adjoint \(a_1,\ldots,a_p\) and unitary elements \(u_1,\ldots,u_p\) in some non-commutative probability space \((\mathcal A,\phi)\) the conjugated elements \(u_1 a_1 u_1^\ast,\ldots,u_pa_pu_p^\ast\) are free whenever \(\set{a_1,\ldots,a_p}\) is \(\ast\)-free from \(\set{u_1,\ldots,u_p}\) and \(\phi(u_iu_j^\ast)=0\) for \(i\ne j\). Note that here freeness among the unitaries is not required, exactly as in Corollary~\ref{corr free indep} where the unitaries \(e^{\ii t_i W},e^{\ii t_j W}\) are not independent. In this context we also mention that for a set \(S\) and a Haar unitary \(u\) in a free probability space the conjugations of \(S\) with respect to different powers of the same unitary \(u\), i.e.\ \(S,uSu^\ast,u^2 S(u^\ast)^2\), are free, see~\cite[Lemma 3.7]{MR1784419}. The estimate~\eqref{eq corr free} shows that for large \(t\) the unitary random matrix \(e^{\ii tW}\) is close to a Haar unitary in this sense.

\subsection{Extensions of Theorem~\ref{theo:intgs}}
We close this section with a few extensions of Theorem~\ref{theo:intgs}. The first one is straight-forward and we illustrate its proof immediately by an example, whilst the other two will be proven in Section~\ref{sec:corex}.
\begin{extension}[Multiple independent Wigner matrices]\label{rmk:manyW}
  Due to the high-probability sense of Theorem~\ref{theo:intgs} we immediately obtain generalisations to multiple independent Wigner matrices both in averaged and isotropic sense by applying Theorem~\ref{theo:intgs}, to resolve
   each Wigner matrix iteratively while conditioning on all others. For example, let \(W,W'\) denote two independent Wigner matrices satisfying Assumption~\ref{ass:entr}. Then, as an example, we obtain 
  \begin{align*}
      \braket{f_1(W)A_1 f_2(W')A_2 f_3(W)A_3} &= \braket{f_2}_\mathrm{sc}\braket{ f_1(W)A_1A_2 f_3(W)A_3 } + \landauO*{\frac{N^\xi}{N}\norm{f_2}_{H^2}} \\
      &= \braket{A_1 A_2 A_3}\braket{f_1}_\mathrm{sc}\braket{f_2}_\mathrm{sc}\braket{f_{3}}_\mathrm{sc} +  \landauO*{\frac{N^\xi}{N}\max_i\norm{f_i}_{H^2}}\\
      &\quad + \braket{A_1A_2}\braket{A_3}\braket{f_2}_\mathrm{sc}\bigl(\braket{f_1f_3}_\mathrm{sc}-\braket{f_1}_\mathrm{sc}\braket{f_3}_\mathrm{sc}\bigr)
  \end{align*}
with very high probability, where in the first step we used Theorem~\ref{theo:intgs}
 for the random matrix \(W'\) after conditioning on \(W\),
while in the second step we used Theorem~\ref{theo:intgs} again for \(W\).  This result
should be compared with corresponding expression for \(W=W'\) in Example~\ref{example theo} which implies 
\[\begin{split}
  &\braket{f_1(W)A_1 f_2(W)A_2 f_3(W)A_3} \\
  &= \braket{A_1 A_2 A_3}\braket{f_1}_\mathrm{sc}\braket{f_2}_\mathrm{sc}\braket{f_2}_\mathrm{sc} + \braket{A_1 A_2} \braket{A_3}\braket{f_2}_\mathrm{sc}\bigl(\braket{f_1f_3}_\mathrm{sc}-\braket{f_1}_\mathrm{sc}\braket{f_3}_\mathrm{sc}\bigr) \\
  &\quad  + \braket{A_1 A_3} \braket{A_2}\braket{f_1}_\mathrm{sc}\bigl(\braket{f_2f_3}_\mathrm{sc}-\braket{f_2}_\mathrm{sc}\braket{f_3}_\mathrm{sc}\bigr)+ \braket{A_2 A_3} \braket{A_1}\braket{f_3}_\mathrm{sc}\bigl(\braket{f_1f_2}_\mathrm{sc}-\braket{f_1}_\mathrm{sc}\braket{f_2}_\mathrm{sc}\bigr)\\
  &\quad + \braket{A_1}\braket{A_2}\braket{A_3} \bigl(\braket{f_1f_2f_3}_\mathrm{sc}-\braket{f_1f_2}_\mathrm{sc}\braket{f_3}_\mathrm{sc}-\braket{f_1f_3}_\mathrm{sc}\braket{f_2}_\mathrm{sc}\\
  &\qquad\qquad\qquad\qquad\qquad\qquad-\braket{f_2f_3}_\mathrm{sc}\braket{f_1}_\mathrm{sc}+2\braket{f_1}_\mathrm{sc}\braket{f_2}_\mathrm{sc}\braket{f_3}_\mathrm{sc}\bigr)+\landauO*{N^{\xi-1}\max_i\norm{f_i}_{H^2}}.
\end{split}\]
Similar statements  hold  for arbitrary number of independent Wigner matrices \(W_1,W_2,W_3,\dots\) with possible repetitions.
\end{extension}

\begin{extension}[Mesoscopic version]\label{rem:mesver}
  For \(0<a<1\) and mesoscopically rescaled \(f_i(x)=g_i(N^{a}(x-E))\) with \(|E|\le 3\) and \(g_i\in H_0^k\) compactly supported, the result in Theorem~\ref{theo:intgs} holds with the bound \(N^{\xi-1}\max_i\norm{g_i}_{H^k}\) replacing the rhs.\ in~\eqref{eq:niceintform}.
\end{extension} 

\begin{extension}[Regularity of test functions]\label{rem:lessreg}
Theorem~\ref{theo:intgs} gives an error bound practically of order \(N^{-1}\) under a relatively high regularity assumption on the functions \(f_i\). Similar result holds for less regular functions with a weaker error bound. For example, if all \(f_i\in H^p([-3,3])\), with some \(p\in [2,k]\), then we obtain~the averaged bound in~\eqref{eq:niceintform} with an error term \(N^{\xi-(p-1)/(k-1)}\prod_i\norm{f_i}_{H^p}\) in the rhs., and the isotropic bound in~\eqref{eq:niceintform} with an error term \(N^\xi[N^{-(p-1)/(2k-3)}\vee N^{-1/2}] \prod_i\norm{f_i}_{H^p}\).
\end{extension} 

\subsection{M\"obius for non-crossing partitions}
Finally, in terms of the Kreweras complement, we provide an explicit expression for the M\"obius function \(\mu(\pi,1_S)\) defined recursively in~\eqref{f circ expl}. Lemma~\ref{mobius expl} is a standard result in the free probability literature but we present the proof for convenience.
\begin{lemma}\label{mobius expl}
  For any finite \(S\subset\N\) and any \(\cP\in\NCP(S)\) we have 
  \[\mu(\cP,1_S)=(-1)^{\abs{\cP}-1} \prod_{B\in K(\cP)} C_{\abs{B}-1},\]
  where \(C_n\) denotes the \(n\)-th Catalan number, i.e.\ \((C_0,C_1,C_2,C_3,\ldots)=(1,1,2,5,14,\ldots )\). 
\end{lemma}
\begin{proof}
  As noticed in~\cite[Proposition 1]{MR1268597} the interval \([\pi,1_S]\) (with respect to the refinement partial order) is isomorphic to products of elementary intervals of the form \([0_B,1_B]\) and therefore the general formula for the M\"obius function follows directly from the special case \(\mu(0_B,1_B)=(-1)^{\abs{B}-1}C_{\abs{B}-1}\), as computed in~\cite[Corollary 5]{MR1268597}. Following e.g.~\cite[Lemma~2.14]{muhle}, this idea can conveniently be presented by using the Kreweras complement, noting that \(K\) is an anti-automorphism in the sense that \(\pi\le \sigma\) if and only if \(K(\pi)\ge K(\sigma)\), and thus we have the isomorphism
  \begin{equation*}
    [\pi,1_S] \cong [K(1_S),K(\pi)] = [0_S,K(\pi)] \cong \prod_{B\in K(\pi)} [0_B,1_B].
  \end{equation*}
  Consequently, since the M\"obius function as defined in~\eqref{f circ expl} is multiplicative, it follows that 
  \[ 
  \begin{split} 
    \mu(\pi,1_S)&=\prod_{B\in K(\pi)}\mu(0_B,1_B)=\prod_{B\in K(\pi)}(-1)^{\abs{B}-1}C_{\abs{B}-1}=(-1)^{\abs{S}-\abs{K(\pi)}}\prod_{B\in K(\pi)}C_{\abs{B}-1} 
  \end{split}
  \]
  and the claim follows from \(\abs{\pi}+\abs{K(\pi)}=\abs{S}+1\). 
\end{proof}

\section{Multi resolvent local laws}\label{sec resolvent}

Before stating the local laws for \(G_1A_1G_2\cdots A_{k-1}G_k\), we introduce the commonly used definition of stochastic domination (see, e.g.~\cite{MR3068390}):

\begin{definition}[Stochastic Domination]\label{def:stochDom}
  If \[
  X=\tuple*{ X^{(N)}(u) \given N\in\N, u\in U^{(N)} }\quad\text{and}\quad Y=\tuple*{ Y^{(N)}(u) \given N\in\N, u\in U^{(N)} }\] are families of non-negative random variables indexed by \(N\), and possibly some parameter \(u\), then we say that \(X\) is stochastically dominated by \(Y\), if for all \(\epsilon, D>0\) we have \[\sup_{u\in U^{(N)}} \Prob\left[X^{(N)}(u)>N^\epsilon  Y^{(N)}(u)\right]\leq N^{-D}\] for large enough \(N\geq N_0(\epsilon,D)\). In this case we use the notation \(X\prec Y\) or \(X= \landauOprec*{Y}\).
\end{definition}

For \(k=1\) it is well known that as \(N\to +\infty\) the resolvent \(G\) is well approximated by the unique solution \(m=m(z)\) of the Dyson equation
\begin{equation}
\label{eq MDE}
-\frac{1}{m}=m+z, \quad \Im m \Im z>0.
\end{equation}
The optimal local law for a single \(G\) is well known~\cite{MR2871147,MR3103909,MR3183577} (see e.g.~\cite[Appendix A]{MR4221653} to extend the local law to \(\eta\ge N^{-100}\)):
\begin{theorem}[Single \(G\) local laws]\label{single G local law}
  Let \(z\in\C\setminus\R\) with \(\eta:=\abs{\Im z}\ge N^{-100}\). Then for deterministic matrices and vectors \(A,\vx,\vy\) with bounded norms \(\norm{A}+\norm{\vx}+\norm{\vy}\lesssim 1\) we have 
  \[\braket{GA}= m\braket{A} + \landauOprec*{\frac{1}{N\eta}}, \quad \braket{\vx,G\vy}=m\braket{\vx,\vy}+\landauOprec*{\sqrt{\frac{\rho}{N\eta}}}\]
  with \(\rho:=\pi^{-1}\abs{\Im m}\). 
\end{theorem}
For \(k\ge 2\) the resolvent identity \(G(z_1)G(z_2)=[G(z_1)-G(z_2)]/[z_1-z_2]\) suggests that \emph{divided differences} of \(m\) provide the deterministic approximation to \(G_1G_2\cdots G_k\), i.e.\ in the case when the deterministic matrices are \(A_1=A_2=\cdots=A_{k-1}=I\). 
\begin{definition}[Divided differences]\label{div diff def}
  For finite multi-sets \(\set{z_1,\ldots,z_n}\subset\C\setminus \mathbf{R}\) we recursively define 
  \begin{equation}\label{div diff dist}
    m[z_1,\ldots,z_n]:= \frac{m[z_2,\ldots,z_n]-m[z_1,\ldots,z_{n-1}]}{z_{n}-z_{1}}
  \end{equation}
  in case there are two distinct \(z_1\ne z_n\) among \(z_1,\ldots,z_n\), and otherwise we set
  \begin{equation}
    m[\underbrace{z,\ldots,z}_{n\text{ times}}] := \frac{m^{(n-1)}(z)}{(n-1)!}.
  \end{equation}
  We note that this is well defined in the sense that \(m[z_1,\ldots,z_n]\) is independent of the ordering of the multi-set \(\set{z_1,\ldots,z_n}\).
\end{definition}  

The main technical result to prove Theorem~\ref{theo:intgs} is the local law for alternating products of resolvents and deterministic matrices, which will be proven later in Section~\ref{sec:proofllaw}.
\begin{theorem}[Multi-resolvent local law]\label{local law}
  Let \(W\) be a Wigner matrix satisfying Assumption~\ref{ass:entr} with resolvent \(G(z)=(W-z)^{-1}\). For \(k\in\N\) let \(z_1,\ldots,z_k\in\C\setminus\R\) be such that $|\Re z_i|\le 3$ and $|\Im z_i|\ge N^{-1}$, and let \(A_1,\ldots,A_k,\vx,\vy\) be arbitrary deterministic matrices and vectors with \(\norm{\vx},\norm{\vy},\norm{A_i}\lesssim1\), and define
  \begin{equation}\label{eq M1..k def}
    M_{[k]} := \sum_{\cP\in \NCP([k])} \pTr_{K(\cP)} (A_1,\ldots, A_{k-1}) \prod_{B\in \cP} m_\circ[B],
  \end{equation}
  where \(m_\circ\) is the free cumulant function from Definition~\ref{m kappa def} of the divided difference \(m[i_1,\ldots,i_n]:=m[z_{i_1},\ldots,z_{i_n}]\) from Definition~\ref{div diff def}. Then with 
  \begin{equation}\label{eq eta rho def}
    \eta_*:=\min_i\abs{\Im z_i},\qquad \rho:=\max_i\rho(z_i), \qquad \rho(z):=\frac{\abs{\Im m(z)}}{\pi} 
  \end{equation} we have 
  \begin{subequations}\label{eq local laws}
    \begin{align}\label{iso local law}
      \braket{\vx,G_1 A_1 G_2 \ldots A_{k-1}G_k\vy} &= \braket{\vx,M_{[k]}\vy} + \landauOprec*{ \frac{1}{\eta_*^{k-1}}\sqrt{\frac{\rho}{N\eta_*}} }, \\\label{av local law}
      \braket{G_1 A_1 \ldots A_{k-1} G_k A_k} &= \braket{M_{[k]}A_k} + \landauOprec*{\frac{1}{N\eta_*^k}}.
    \end{align}
  \end{subequations} 
\end{theorem}
The error estimates in~\eqref{iso local law}--\eqref{av local law} are optimal. In fact, elementary calculations show that e.g.\ for GUE \(W\) we have 
\begin{equation}\label{single G var}
  \begin{split}
    \sqrt{\E \abs{\braket{G-m}}^2} &= \frac{1}{N}\frac{\Im m}{\eta}\frac{1}{\abs{1-m^2}}\Bigl(1+\landauO*{\frac{1}{N \eta}}\Bigr) \sim \frac{1}{N\eta}, \\ 
    \sqrt{\E \abs{\braket{\vx,(G-m)\vy}}^2} &= \norm{\vx}\norm{\vy}\sqrt{\frac{\Im m\abs{m}}{N\eta}}\Bigl(1+\landauO*{\frac{1}{N \eta}}\Bigr)\sim\sqrt{\frac{\rho}{N\eta}},
  \end{split}
\end{equation}
demonstrating the optimality for \(k=1\). For several \(G\)'s we have by the resolvent identity \(GG^\ast=\Im G/\Im z\), and thus the optimality for \(k\ge 2\) follows from~\eqref{single G var}. 
\begin{remark}
  If the imaginary parts \(\Im z_i\) vary in size, and the intermediate matrices are simply identity matrices, \(A_i=I\), then the error bounds in~\eqref{eq local laws} can be improved to 
  \begin{equation}\label{multi G no A}
    \begin{split}
      \braket{\vx,G_1G_2\cdots G_k\vy} &= \braket{\vx,\vy}m[z_1,\ldots,z_k] + \landauO*{ \frac{1}{\prod_i \abs{\Im z_i}} \sqrt{\frac{\min_i \abs{\Im z_i}}{N}}  },\\
      \braket{G_1G_2\cdots G_k} &= m[z_1,\ldots,z_k] + \landauO*{ \frac{1}{N\prod_i \abs{\Im z_i}} },
    \end{split}
  \end{equation} 
  by using the improved bound from~\cite[Theorem 3.5]{2012.13218} instead of~\eqref{av underline}--\eqref{iso underline} later. However, in the absence of deterministic matrices the local law~\eqref{multi G no A} can alternatively also be derived from integrating the single-\(G\) local law, see~\cite[Lemma 3.9]{2103.05402}.
\end{remark}
\begin{example}We consider some examples. 
  \begin{enumerate}[label=(\roman*)]
    \item For \(k=2\) there are only two non-crossing partitions \(\set{12},\set{1|2}\) and thus 
    \[M_{[2]}= A_1 m_\circ[1]m_\circ[2]+\braket{A_1} m_\circ[1,2]=A_1 m[z_1]m[z_2]+\braket{A_1}(m[z_1,z_2]-m[z_1]m[z_2]).\]
    \item For \(k=3\) all non-crossing partitions are given in Figure~\ref{ex3} and thus we obtain 
    \begin{equation*}
      \begin{split}
        M_{[3]} &= A_1 A_2 m_\circ[1]m_\circ[2] m_\circ[3]  + A_2\braket{A_1} m_\circ[3] m_\circ[1,2] + \braket{A_1 A_2} m_\circ[2] m_\circ[1,3] \\
        &\qquad + A_1\braket{A_2} m_\circ[1] m_\circ[2,3] + \braket{A_1}\braket{A_2} m_\circ[1,2,3]
      \end{split}
    \end{equation*}
    with \(m_\circ[i]=m[z_i]\), \(m_\circ[i,j]=m[z_i,z_j]-m[z_i]m[z_j]\) and 
    \[ m_\circ[i,j,k] = m[z_i,z_j,z_k] - m[z_i]m[z_j,z_k] - m[z_j]m[z_i,z_k] -m[z_k]m[z_i,z_j] + 2 m[z_i]m[z_j]m[z_k]  \]
    due to~\eqref{f circ expl}.
    \item For \(A_1=\cdots=A_{k-1}=I\) we have \(\pTr_\cP(A_1,\ldots,A_{k-1})=I\) for any \(\cP\) and thus
    \[M_{[k]}=\sum_{\cP\in\NCP([k])}\prod_{B\in\cP}m_\circ[B]= m[z_1,\ldots,z_k]\]
    due to~\eqref{eq m kappa def}. 
  \end{enumerate}
\end{example}

The deterministic approximations \(M_{[k]}\) satisfy the following bounds (which will be proven in Section~\ref{sec:addres}).

\begin{lemma}\label{lemma M bound}
  For any \(k\ge1\) with \(\eta_*,\rho\) as in~\eqref{eq eta rho def} we have the bound 
  \begin{equation}\label{norm M bound} 
    \norm{M_{[k]}} \lesssim \frac{\rho}{\eta_*^{k-1}}\prod_{j=1}^{k-1}\norm{A_j},
  \end{equation}
  and if \(\mathfrak{a}\) out of the matrices \(A_1, \ldots, A_k\) are traceless, then we also have 
  \begin{equation}\label{tr M bound}
    \abs{\braket{M_{[k]}A_{k}}} \lesssim \frac{\rho}{\eta_*^{k-1-\lceil \mathfrak{a}/2\rceil}}\prod_{j=1}^{k}\norm{A_j}.
  \end{equation}
\end{lemma}

In fact, for generic matrices $A_i$ and $z_1=\dots=z_k$, with $\eta_*=\Im z_i$, these bounds are optimal, which shows that error terms in~\eqref{iso local law}--\eqref{av local law} are typically smaller than the deterministic leading terms whenever \(N\eta_*\rho\gg 1\), i.e.\ in the regime where \(\eta\) is larger than the local eigenvalue spacing.

\begin{remark}
  The bound~\eqref{norm M bound} is consistent with the corresponding bound on resolvent chains 
  \[\abs{\braket{\vx,G_1A_1\ldots A_{k-1}G_k\vy}} \lesssim \frac{\sqrt{\braket{\vx,\Im G_1\vx} \braket{\vy,\Im G_k\vy} }}{\eta_*^{k-1}}\le \frac{\rho}{\eta_*^{k-1}} \]
  obtained via Cauchy-Schwarz, the Ward identity 
  \begin{equation}
  \label{eq:ward}
  GG^\ast=\frac{\Im G}{\Im z},
    \end{equation}
   and the norm bound \(\norm{G}\le 1/\abs{\Im z}\). Similarly,~\eqref{tr M bound} in the critical \(\eta_*\sim N^{-1+\epsilon}\) regime is consistent with the recently established asymptotic orthogonality~\cite[Theorem~2.2]{2012.13215} of the eigenvectors \(\set{\bm u_i}\) of \(W\) with respect to traceless observables \(\braket{A}=0\) in the sense \(\abs{\braket{\bm u_i,A\bm u_j}}\prec N^{-1/2}\). Indeed, by spectral decomposition it follows that 
  \[ \abs{\braket{G_1A_1\ldots G_k A_k}} \le \frac{1}{N} \sum_{a_1\cdots a_k} \frac{\abs{\braket{\bm u_{a_1},A_1\bm u_{a_2}}}\cdots \abs{\braket{\bm u_{a_k},A_k\bm u_{a_1}}}}{\abs{\lambda_{a_1}-z_1}\cdots \abs{\lambda_{a_k}-z_k}} \prec N^{k-1-\mathfrak{a}/2}\]
  in the case when \(\mathfrak{a}\) of the matrices \(A_1,\ldots,A_k\) are traceless. 
\end{remark}

\section{Proof of the main results}
\label{sec:prmares}


Using the Helffer-Sj\"ostrand representation, we express \(f_1(W)A_1\cdots f_k(W)\) as an integral of products of resolvents at different spectral parameters. Without loss of generality may assume (see Section~\ref{sec:mresnew} for details) 
that $f_i\in H^k_0([-3,3])$ and we consider such functions naturally extended to the entire real line
by setting it zero outside of $[-3,3]$. 

We now recall the Helffer-Sj\"ostrand representation. For any 
 \(f\in H^k_0([-3,3])\)  
 we define its almost analytic extension by
\begin{equation}
\label{eq:almostanalest}
f_\mathbf{C}(z)=f_{\mathbf{C},k}(z)=f_{\mathbf{C},k}(x+\ii\eta):=\left[\sum_{j=0}^{k-1} \frac{(\ii\eta)^j}{j!} f^{(j)}(x)\right]\chi(\eta),
\end{equation}
with \(\chi(\eta)\) a smooth cut-off equal to one on \([-5, 5]\) and equal to zero on \([-10,10]^c\). Then by Helffer-Sj\"ostrand representation~\cite{MR1345723}, we have
\begin{equation}
\label{eq:HS}
f(\lambda)=\frac{1}{\pi}\int_\mathbf{C}\frac{\partial_{\overline{z}}f_\mathbf{C}}{\lambda-z}\,\dif^2 z,
\end{equation}
where \(\dif^2 z=\dif x\dif \eta\) denotes the Lebesgue measure on \(\mathbf{C}\equiv\mathbf{R}^2\) with \(z=x+\ii\eta\), and $\partial_{\overline{z}}:=(\partial_x+\ii \partial_\eta)/2$.

By~\eqref{eq:HS} we get
\begin{equation}\label{eq:intrep}
f_1(W)A_1\cdots f_k(W)=\frac{1}{\pi^k}\int_{\C^k} \prod_{i=1}^k\dif^2z_i\left[\prod_{i=1}^k (\partial_{\bar z}(f_i)_\C)(z_i)\right] G(z_1)A_1\cdots A_{k-1}G(z_k),
\end{equation}
where \(G(z_i):=(W-z_i)^{-1}\).
In particular, using~\eqref{eq:intrep}, we reduce the analysis of~\eqref{eq:niceintform} to proving a \emph{local law} for alternating chains of resolvents \(G_i=G(z_i)\) and bounded deterministic matrices \(A_i\), i.e.\ for 
\begin{equation}
\label{GAGAG}
G_1A_1G_2\cdots A_{k-1}G_k.
\end{equation}
See Theorem~\ref{local law} for the precise statement.

In Section~\ref{sec:mresnew} we prove our main result Theorem~\ref{theo:intgs}, then in Section~\ref{sec:corex} we will prove some of its corollaries and extensions. Finally, in Section~\ref{sec:addres} we prove some additional technical results used within the proof of Theorem~\ref{theo:intgs}.

\subsection{Proof of Theorem~\ref{theo:intgs}}
\label{sec:mresnew}

Set \(H^k:=H^k([-3,3])\).  If \(\max_i \norm{f_i}_{H^k}\ge N^{1-\xi}\), with \(\xi>0\) from the statement of Theorem~\ref{theo:intgs},  then there is nothing to prove since the lhs.\ of~\eqref{eq:niceintform} is clearly bounded by \(\prod_i \norm{f_i}_{L^\infty}\lesssim 1\). In the remainder of the proof we can thus assume that \(\norm{f_i}_{H^k}\le N^{1-\xi}\), which implies \(\norm{f^{(p)}}_{L^2}\le N^{1-\xi}\) for any \(1\le p\le k\). Additionally, without loss of generality we can assume that \(f\in H_0^k([-3,3])\), indeed if this is not the case is then it is enough to consider \(f_\chi(x):=f(x)\chi(x)\), with \(\chi\) a smooth cut-off function which is equal to one on \([-5/2,5/2]\) and equal to zero on \([-3,3]^c\). Then, by eigenvalue rigidity (see e.g.~\cite[Theorem 7.6]{MR3068390} or~\cite{MR2871147}), \(f_\chi(W)=f(W)\) with very high probability. 
 Furthermore, we consider any function  \(f\in H_0^k([-3,3])\) to be extended to $\R$ by zero outside of $[-3, 3]$.
We first prove the average case in~\eqref{eq:niceintform}, and then we explain the very minor changes required in the isotropic case.





In the following computations we will often use the bound 
\begin{equation}
\label{eq::bder}
\int_\mathbf{R} \dif x |\partial_{\overline{z}} f_{\mathbf{C},k}(x+\ii\eta)|\lesssim \eta^{k-1}\norm{f}_{H^k}
\end{equation}
from~\eqref{eq:almostanalest}. We now prove that the regime \(|\eta_i|\ge \eta_0\), for some \(i\in [k]\) and with \(\eta_0:=N^{-1+\xi/2}\), in the integral representation of \(\braket{f_1(W)A_1\dots f_k(W)A_k}\) from~\eqref{eq:intrep} is negligible. From now on we use the notation $f_\mathbf{C}(z)=f_{\mathbf{C},k}(z)$. We first consider the case when only a single \(|\eta_i|\le \eta_0\) and then we explain the minor changes in the case more then one \(\eta_i\)'s are small. Without loss of generality we assume that \(|\eta_1|\le \eta_0\); in this regime we will prove that 
\begin{equation}
\label{eq:prelred}
\begin{split}
&\Bigg|\int \dif x_1\cdots \dif x_k\int_{\substack{|\eta_i|\ge \eta_0,\\ i\in [2,k]}}\dif \eta_2\cdots\dif \eta_k \int_{-\eta_0}^{\eta_0} \dif \eta_1  \left( \prod_{i=1}^k(\partial_{\bar z}(f_i)_\C)(z_i)\right)\braket{G(z_1)A_1\cdots G(z_k)A_k}\Bigg|\\
&\qquad\prec \eta_0\max_i\norm{f_i}_{H^k},
\end{split}
\end{equation}



To prove the bound~\eqref{eq:prelred} we will use Stokes theorem in the following form:
\begin{equation}
\label{eq:stokes}
\int_{-10}^{10}\int_{\widetilde{\eta}}^{10} \partial_{\overline{z}}\psi(x+\ii\eta)h(x+\ii\eta)\, \dif \eta\dif x=\frac{1}{2\ii}\int_{-10}^{10}\psi(x+\ii\widetilde{\eta})h(x+\ii\widetilde{\eta})\, \dif x,
\end{equation}
for any \(\widetilde{\eta}\in[0,10]\), and for any \(\psi,h\in H^1(\mathbf{C})\equiv H^1(\mathbf{R}^2)\) such that \(\partial_{\overline{z}}h=0\) on the domain of integration and for \(\psi\) vanishing at the left, right and top boundary of the domain of integration. Note that by~\eqref{eq:stokes} we readily conclude that
\begin{equation}
\label{eq:ineeduse}
\begin{split}
&\int_\mathbf{R} \dif x_i \int_{\eta_0}^{10} \dif \eta_i  (\partial_{\overline{z}}(f_i)_\mathbf{C})(z_i) \braket{G(z_1)A_1\dots A_{i-1}G(z_i)A_i\dots G(z_k)A_k}\\
&\qquad=\int_\mathbf{R} \dif x_i (f_i)_\mathbf{C}(x_i+\ii\eta_0) \braket{G(z_1)A_1\dots A_{i-1}G(x_i+\ii\eta_0)A_i\dots G(z_k)A_k},
\end{split}
\end{equation}
for any fixed \(z_1,\dots, z_{i-1},z_{i+1},\dots,z_k\). Then, using~\eqref{eq:ineeduse} repeatedly for the \(z_2,\dots,z_k\)-variables, and defining \(\eta_r:=N^{-100}\), we conclude
\begin{equation}
\label{eq:intbpbettb}
\begin{split}
|\text{lhs.\ of~\eqref{eq:prelred}}|&=\Bigg|\int \prod_{i=1}^k \dif x_i \int_{-\eta_0}^{\eta_0} \dif \eta_1 (\partial_{\overline{z}}(f_1)_\mathbf{C})(x_1+\ii\eta_1)\prod_{i=2}^k (f_i)_\mathbf{C}(x_i+\ii\eta_0) \\
&\qquad\qquad\qquad\qquad\quad\times \braket{G(z_1)A_1G(x_2+\ii\eta_0)\cdots G(x_k+\ii\eta_0)A_k}|\Bigg| \\
&\prec \norm{f_1}_{H^k} \left(\int_{\eta_r\le|\eta_1|\le \eta_0}\left(\eta_1^{k-3/2}+\frac{\eta_1^{k-5/2}}{N}\right)\frac{1}{\eta_0^{k-3/2}}\dif \eta_1 +\int_{|\eta_1|\le \eta_r}\eta_1^{k-2}\eta_0^{-k+1}  \dif \eta_1\right) \\
&\lesssim \eta_0\norm{f_1}_{H^k}.
\end{split}
\end{equation}
In the first inequality  in the regime \(\eta_r = N^{-100}\le|\eta_1|\le \eta_0\) we used~\eqref{eq::bder} for
 \(\int \dif x_1 |\partial_{\overline{z}}(f_1)_\mathbf{C}|\), and the bound 
\begin{equation}
\label{eq:needisoaswell}
\begin{split}
&\braket{G(z_1)A_1G(x_2+\ii\eta_0)\cdots G(x_k+\ii\eta_0)A_k} \\
&\le \braket{G(z_1)A_1A_1^*G(z_1)^*}^{1/2}\braket{G(x_2+\ii\eta_0)A_2\cdots G(x_k+\ii\eta_0)A_k A_k^*G(x_k+\ii\eta_0)^*\cdots
A_2^* G(x_2+\ii\eta_0)^*}^{1/2} \\
&\prec\frac{1}{\sqrt{\eta_1\eta_0}\eta_0^{k-2}}\braket{\Im G(z_1)}^{1/2}\braket{\Im G(x_2+\ii \eta_0)}^{1/2} \\
&\prec \frac{1}{\sqrt{\eta_1\eta_0}\eta_0^{k-2}}\left(1+\frac{1}{N\eta_1}\right).
\end{split}
\end{equation}
Here, to go from first to the second line we used a Schwarz inequality, to go from the second to the third line we used the norm bounds 
$\norm{A_1A_1^*}\lesssim 1$ and 
\[
\norm{A_2\cdots G(x_k+\ii\eta_0)A_kA_k^*G(x_k+\ii\eta_0)^*\cdots
A_2^* }\lesssim \eta_0^{-2k+4}
\]
 and Ward identity \eqref{eq:ward}. Finally, 
  to go to the last line we used the averaged local law from Theorem~\ref{single G local law} to show the boundedness
  of $\braket{\Im G}$.  In the complementary regime \(|\eta_1|\le N^{-100}\) we used the norm bound
\[
\big|\braket{G(z_1)A_1G(x_2+\ii\eta_0)A_2\cdots G(x_k+\ii\eta_0)A_k}\big|\le \norm{G(z_1)A_1}\prod_{i=2}^k \norm{G(x_i+\ii\eta_0)A_i} \lesssim \frac{1}{\eta_1\eta_0^{k-1}},
\]
together with the estimate~\eqref{eq::bder} for \(\partial_{\overline{z}}(f_1)_\mathbf{C}\). Note that in the penultimate inequality of~\eqref{eq:intbpbettb} we also used that
\[
\norm{(f_i)_\mathbf{C}(\cdot+\ii\eta_0)}_{L^1}\lesssim\norm{(f_i)_\mathbf{C}}_{L^2}\lesssim \sum_{j=0}^{k-1}\frac{\eta_0^j}{j!} \norm{f_i^{(j)}}_{L_2}\lesssim 1,
\]
by~\eqref{eq:almostanalest} and \(\norm{f_i}_{H^k}\le N^{1-\xi}\).




The bound for the regime when more than one \(\eta_i\) are smaller than \(\eta_0\) (in absolute value) is completely analogous giving an even smaller bound. In particular, if \(|\eta_1|,\ldots, |\eta_l|\le \eta_0\), with \(l\in [k]\), then we perform an integration by parts in the \(z_{l+1}, \ldots,z_k\)-variables and use the bound~\eqref{eq::bder} for \(\partial_{\overline{z}}(f_i)_\mathbf{C}\) for \(i\in [l]\) (see~\eqref{eq:intbpbettb} below for the case \(l=1\)), which gives a bound
\[
\eta_0^{(l-1)k+1}\left(\prod_{i=1}^l \norm{f_i}_{H^k}\right)\le \eta_0^{(l-1)k+1}N^{l-1}\max_i\norm{f_i}_{H^k}
\]
by \(\norm{f_i}_{H^k}\le N^{1-\xi}\). This concludes the proof of the fact that the small \(\eta_i\) regime is negligible. We now estimate the regime when \(|\eta_i|\ge \eta_0\) for any \(i\in [k]\).



By~\eqref{eq:intrep} and the local law~\eqref{av local law}, we conclude that
\begin{equation}
\label{eq:mintegralformhs}
\begin{split}
&\braket{f_1(W)A_1\cdots f_k(W)A_k} \\
&\qquad=\frac{1}{\pi^k}\int_{\R^k}\int_{\eta_0\le |\eta_i|\le 10}\dif^2z_1\cdots \dif^2z_k (\partial_{\bar z}(f_1)_\C)(z_1)\cdots (\partial_{\bar z}(f_k)_\C)(z_k)\braket{M_{[k]}A_k} \\
&\qquad\quad+\mathcal{O}_\prec\left(\eta_0\max_i\norm{f_i}_{H^k}\right).
\end{split}
\end{equation}
Note that in~\eqref{eq:mintegralformhs}, proceeding as in~\eqref{eq:intbpbettb}, we estimated the error term coming from the local law~\eqref{av local law} by
\begin{equation}
\label{eq:mintegralformhs2}
\begin{split}
&\frac{1}{\pi^k}\int_{\R^k}\int_{\eta_0\le |\eta_i|\le 10}\dif^2{\bm z} \prod_{i=1}^k(\partial_{\bar z}(f_i)_\C)(z_i)\braket{(G(z_1)A_1\dots G(z_k)-M_{[k]})A_k} \\
&\qquad\qquad\qquad\quad=\mathcal{O}_\prec\left(N^{-1}\max_i\norm{f_i}_{H^k}\right),
\end{split}
\end{equation}
with \(\dif^2{\bm z}:=\dif^2 z_1\dots \dif^2 z_k\). More precisely, in~\eqref{eq:mintegralformhs2} we considered the regime \(\eta_1\le \eta_2\le \cdots \le \eta_k\) (all the other regimes give the same contribution by symmetry) and performed \(k-1\) integration by parts in the \(z_i\)-variables, \(i\in [2,k]\), as in~\eqref{eq:ineeduse}, and then estimated the remaining \(\partial_{\overline{z}}(f_1)_\mathbf{C}(z_1)\) by~\eqref{eq::bder}. Note that the factors $N^{\xi/2}$ from $\eta_0=N^{-1+\xi/2}$, and $|\log \eta_0|$ from the integration of the error term in the local law were all included in the $\mathcal{O}_\prec(\cdot)$ notation.
 


By the formula for \(M_{[k]}\) in~\eqref{eq M1..k def} and the definition of \(m_\circ\) from Definition~\ref{m kappa def}, to compute the rhs.\ of~\eqref{eq:mintegralformhs} it is enough to compute the integral of \(m[z_1,\ldots,z_p]\) for any \(p\in\mathbf{N}\). This technical lemma will be proven at the end of this section.
\begin{lemma}\label{lem:divdiffinths}
For any \(p\in \N\) denote \({\bm z}:=(z_1,\ldots, z_p)\in\mathbf{C}^p\), then it holds
\begin{equation}\label{eq:smartinths}
\frac{1}{\pi^p}\int_{\R^p}\int_{\eta_r\le |\eta_i|\le 10}\dif^2 {\bm z}\, \prod_{i=1}^p (\partial_{\overline{z}} (f_i)_\mathbf{C})(z_i)  m[z_1,\ldots,z_p]=\braket*{\prod_{i=1}^p f_i}_{\mathrm{sc}}+\mathcal{O}(\eta_r),
\end{equation}
where \(\eta_r:=N^{-100}\).
\end{lemma}

Finally, using that by~\eqref{eq:prelred} and Lemma~\ref{lemma M bound} the regime \(\eta_i\in [\eta_r,\eta_0]\) can be added back to~\eqref{eq:mintegralformhs} at the price of an error \(\eta_0\max_i\norm{f_i}_{H^k}\), and using~\eqref{eq:smartinths} repeatedly together with the definition of \(\mathrm{sc}_\circ\) given in the statement of Theorem~\ref{theo:intgs}, we conclude the proof of the average case in~\eqref{eq:niceintform}, modulo the proof of Lemma~\ref{lem:divdiffinths}.

The proof of the isotropic case in~\eqref{eq:niceintform} is very similar. The only differences are the following: (i) to bound the small \(\eta_i\)-regime we have to replace~\eqref{eq:needisoaswell} by
\[
\big|\braket{{\bm x}, G(z_1)A_1G(x_2+\ii\eta_0)\cdots G(x_k+\ii\eta_0) {\bm y}}\big|\prec \frac{1}{\sqrt{\eta_1\eta_0}\eta_0^{k-2}}\left(1+\frac{1}{\sqrt{N\eta_1}}\right),
\]
which still gives exactly the same bound~\eqref{eq:prelred}; (ii) to estimate the error term coming from the isotropic local law~\eqref{iso local law} (used in the regime when \(|\eta_i|\ge \eta_0\) for all \(i\in [k]\)) we have to replace~\eqref{eq:mintegralformhs2} by
\begin{equation}
\label{eq:finbas}
\begin{split}
&\frac{1}{\pi^k}\int_{\R^k}\int_{\eta_0\le |\eta_i|\le 10}\dif^2{\bm z} \prod_{i=1}^k(\partial_{\bar z}(f_i)_\C)(z_i)\braket{{\bm x},(G(z_1)A_1\dots G(z_k)-M_{[k]}) {\bm y}} \\
&\qquad\qquad\qquad\quad=\mathcal{O}_\prec\left(N^{-1/2}\max_i\norm{f_i}_{H^k}\right).
\end{split}
\end{equation}
The proof of~\eqref{eq:finbas} is exactly the same as the proof of~\eqref{eq:mintegralformhs2}.\qed

\subsection{Proof of the corollaries and extensions of Theorem~\ref{theo:intgs}}
\label{sec:corex}

\begin{proof}[Proof of Corollary~\ref{corr:intgs}]
  The leading term~\eqref{eq:niceintform corr} is given exactly by the leading term in~\eqref{eq:niceintform} choosing \(f_i(x)=e^{\ii s_i x}\), and using the definition of \(\varphi(x)\) in~\eqref{eq:defphi}. The bound for the error term readily follows from \(\norm{f_i}_{H^k}\lesssim |s_i|^k\) and \(\norm{f_i}_{L^\infty}\lesssim 1\).
\end{proof}

\begin{proof}[Proof of Corollary~\ref{cor:misscor}]
  For the proof of~\eqref{eq therm s} we note that in~\eqref{eq:niceintform corr} only \(\pi\in\NCP[k]\) contribute for which all blocks of \(K(\pi)\) contain at least two elements. However, this can only be the case if \(\pi\) contains at least two singleton blocks and the claim follows since \(\varphi_\circ[\set{i}]\lesssim (1+\abs{s_i})^{-3/2}\). 
\end{proof}

\begin{proof}[Proof of Corollary~\ref{cor therm}]

  For the proof of~\eqref{eq therm a} we note that due to the ordering of times it follows that 
  \begin{equation}
    \label{eq:decphi}
      \varphi_\circ[B]= \begin{cases}
        \landauO{(1+\min_i\abs{s_i})^{-3/2}}, & B\subsetneq[k],\\
        1, & B=[k].
      \end{cases} 
  \end{equation}
  For \(a\) traceless \(A_i\)'s only partitions \(\pi\) for which \(K(\pi)\) has at most \(\lfloor \mathfrak{a}/2\rfloor\) blocks contribute, i.e.\ only partitions \(\pi\) with at least \(\lceil \mathfrak{a}/2\rceil+1\) blocks and also~\eqref{eq therm a} follows. For the proof of~\eqref{eq therm a iso} we similarly note that only partitions contribute for which \(K(\pi)\) as at most \(\lfloor \mathfrak{a}/2\rfloor+1\) blocks (since the block containing \(k\) has no trace restriction).  
\end{proof}
\begin{proof}[Proof of Corollary~\ref{corr free indep}]
  By linearity it is clearly sufficient to check that for \(k\ge 2\), \(i_1\ne i_2\ne \ldots\ne i_k\in[p]\) and any \(a_i\in\N\) we have 
\begin{equation}\label{asymp freeness}
  \braket*{\Bigl(A_{i_1}(t_{i_1})^{a_1}-\braket{A_{i_1}(t_{i_1})^{a_{1}}}\Bigr)\cdots\Bigl(A_{i_k}(t_{i_k})^{a_k}-\braket{A_{i_k}(t_{i_k})^{a_k}}\Bigr)}  =\landauo{1}.
\end{equation}
This is indeed the case for \(1\ll\abs{t_1-t_2}\ll N^{1/k-\epsilon}\), since the lhs.\ of~\eqref{asymp freeness} simplifies to 
\[\braket*{(A_{i_1}^{a_1})^\circ(t_{i_1})\cdots (A_{i_k}^{a_k})^\circ(t_{i_k}) } = \landauO*{N^\xi \frac{\abs{t_1-t_2}^k }{N}+(1+\abs{t_1-t_2})^{-3} },\]
where \(A^\circ:=A-\braket{A}\) denotes the traceless part of \(A\), and we used~\eqref{eq therm}. 
\end{proof}

\begin{proof}[Proof of the bound in Extension~\ref{rem:mesver}]
The proof of this bound is completely analogous to the proof of~\eqref{eq:niceintform}, the only difference is that instead of~\eqref{eq:almostanalest} we consider the almost analytic extension
\[
f_{\mathbf{C},k}(z)=f_{\mathbf{C},k}(x+\ii\eta):=\left[\sum_{j=0}^{k-1} \frac{(\ii\eta)^j}{j!} f^{(j)}(x)\right]\chi(N^a\eta),
\]
for each \(f=f_i\), where \(f_i=g_i(N^a(x-E))\) and \(a\in (0,1)\), \(|E|<2\).
\end{proof}

\begin{proof}[Proof of the bound in Extension~\ref{rem:lessreg}]
The proof of this bound is similar to the proof of Theorem~\ref{theo:intgs}, the only difference is that the regime \(|\eta_i|\le \eta_0\) and the error term in the local law are estimated differently. Similarly to the proof of Theorem~\ref{theo:intgs}, we prove the bound in Extension~\ref{rem:lessreg} in the average case and then we explain the very minor changes in the isotropic case. Without loss of generality we assume that $f_i\in H_0^p([-3,3])$. We first show how to bound the small \(\eta_i\)-regimes; here we again only consider the case when only \(|\eta_1|\le \eta_0\). Recall that for \(H^p\) functions we have
\begin{equation}
\label{eq:newwroseb}
\int_\mathbf{R} \dif x_i |\partial_{\overline{z}}(f_i)_\mathbf{C}|\lesssim \eta_i^{p-1}\norm{f_i}_{H^p}
\end{equation}
by~\eqref{eq::bder}, where we used the short-hand notation $f_\mathbf{C}(z)=f_{\mathbf{C},p}(z)$. Then, using this bound, we get
\begin{equation}
\label{eq:smallboundopt}
\left|\prod_{i=1}^k \int \dif x_i\prod_{i=2}^k\int_{|\eta_i|\ge \eta_0}\dif \eta_i \int_{-\eta_0}^{\eta_0} \dif \eta_1 \prod_{i=1}^k(\partial_{\bar z}(f_i)_\C)(z_i)\braket{G(z_1)A_1\cdots G(z_k)A_k}\right|\prec \eta_0^{p-1}\prod_i \norm{f_i}_{H^p},
\end{equation}
for some \(N^{-1}\ll \eta_0\ll 1\) that we will choose shortly. In the estimate~\eqref{eq:smallboundopt} we also used the norm bound \(|\braket{G(z_1)A_1\cdots G(z_k)A_k}|\lesssim \prod_i |\eta_i|^{-1}\). Similarly to~\eqref{eq:mintegralformhs2}, using the bound~\eqref{eq:newwroseb} and the average local law~\eqref{av local law}, we conclude that
\begin{equation}
\label{eq:bounderrlllaw}
\left|\prod_{i=1}^k \int \dif x_i\prod_{i=1}^k\int_{|\eta_i|\ge \eta_0}\dif \eta_i  \prod_{i=1}^k(\partial_{\bar z}(f_i)_\C)(z_i)\braket{(G(z_1)A_1\dots G(z_k)-M_{[k]})A_k}\right|\prec \frac{1}{N\eta_0^{k-p}}\prod_i \norm{f_i}_{H^p}
\end{equation}
Optimising the bounds in~\eqref{eq:smallboundopt} and~\eqref{eq:bounderrlllaw} we find that \(\eta_0=N^{-1/(k-1)}\), concluding the proof of the bound stated in Extension~\ref{rem:lessreg} in the average case. The proof in the isotropic case is exactly the same, the only difference is that~\eqref{eq:bounderrlllaw} has to be replaced by
\begin{equation}
\label{eq:basthopd}
\begin{split}
&\left|\prod_{i=1}^k \int \dif x_i\prod_{i=1}^k\int_{|\eta_i|\ge \eta_0}\dif \eta_i  \prod_{i=1}^k(\partial_{\bar z}(f_i)_\C)(z_i)\braket{{\bm x}, (G(z_1)A_1\dots G(z_k)-M_{[k]}) {\bm y}}\right| \\
&\qquad\quad\prec \frac{1}{\sqrt{N}} \left(\frac{1}{\eta_0^{k-p-1/2}}\vee 1\right)\prod_i \norm{f_i}_{H^p},
\end{split}
\end{equation}
where we used the isotropic local law~\eqref{iso local law}. Optimising the error terms in~\eqref{eq:smallboundopt} and~\eqref{eq:basthopd} we conclude that in the isotropic case \(\eta_0=N^{-1/(2k-3)}\).
\end{proof}

\subsection{Proof of additional results used within the proof of Theorem~\ref{theo:intgs}}
\label{sec:addres}

\begin{proof}[Proof Lemma~\ref{lem:divdiffinths}]
  We claim that
  \begin{equation}
  \label{eq:veruseq}
  m[z_1,\ldots,z_k]=\int_\mathbf{R} \rho(x) \prod_{i=1}^p \frac{1}{(x-z_i)} \dif x.
  \end{equation}
  The proof of~\eqref{eq:veruseq} follows by induction. For \(p=1\)~\eqref{eq:veruseq} is trivial, for \(p=2\), we have
  \[
  m[z_1,z_2]=\frac{m[z_1]-m[z_2]}{z_1-z_2}=\int \frac{\rho(x)}{z_1-z_2} \left[\frac{1}{x-z_1}-\frac{1}{x-z_2}\right]\,\dif x=\int \frac{\rho(x)}{(x-z_1)(x-z_2)} \dif x.
  \]
  Now assume that~\eqref{eq:veruseq} holds for \(p\), then it holds for \(p+1\) as well:
  \[
  \begin{split}
  m[z_1,z_2, z_3,\ldots, z_{p+1}]&=\frac{m[z_1,z_3,\ldots,z_{p+1}]-m[z_2,z_3,\ldots, z_{p+1}]}{z_1-z_2} \\
  &=\int \frac{\rho(x)}{z_1-z_2} \left[\frac{1}{x-z_1}-\frac{1}{x-z_2}\right] \prod_{i=3}^p \frac{1}{(x-z_i)}\, \dif x \\
  &=\int \rho(x) \prod_{i=1}^{p+1} \frac{1}{(x-z_i)} \dif x,
  \end{split}
  \]
  concluding the proof of~\eqref{eq:veruseq}. Finally, using~\eqref{eq:veruseq} we readily conclude~\eqref{eq:smartinths}, where we used that for any fixed \(x\in\R\) we have
  \[
  \begin{split}
  \int_\mathbf{R}\dif x_i\int_{\eta_r\le |\eta_i|\le 10}\dif \eta_i (\partial_{\overline{z}} (f_i)_\mathbf{C}(z_i))\frac{1}{x-z} &=\int_\mathbf{R}\dif x_i (f_i)_\mathbf{C}(x_i+\ii\eta_r)\frac{\eta_r}{(x-x_i)^2+\eta_r^2} \\
  &=\pi f(x)+\mathcal{O}(\eta_r),
  \end{split}
  \]
  We remark that in the first equality we used~\eqref{eq:stokes}.
\end{proof}

\begin{proof}[Proof of Lemma~\ref{lemma M bound}]
  We first claim that for \(k\ge 2\)
  \begin{equation}\label{m bound}
    \abs{m[z_1,\ldots,z_k]}\le \eta_\ast^{1-k} \max_i \abs{\Im m(z_i)} 
  \end{equation}
  for any \(z_1,\ldots,z_k\). The bound~\eqref{m bound} follows immediately from~\eqref{eq:veruseq} and estimating 
  \[ 
  \begin{split}
    \abs{m[z_1,\ldots, z_k]}&\le \int\rho_\mathrm{sc}(x) \prod_{i=1}^k \frac{1}{\abs{x-z_i}}\dif x \\
    &\le \frac{\eta_\ast^{2-k}}{2} \int \rho_\mathrm{sc}(x) \Bigl(\frac{1}{\abs{x-z_1}^2}+\frac{1}{\abs{x-z_2}^2}\Bigr)\dif x \\
    &= \frac{\eta_\ast^{2-k}}{2}\Bigl( \frac{\abs{\Im m(z_1)}}{\abs{\Im z_1}}+\frac{\abs{\Im m(z_2)}}{\abs{\Im z_2}}\Bigr)\le \eta_\ast^{1-k} \max_i \abs{\Im m(z_i)}.
  \end{split}\]

  From~\eqref{m bound} it follows that for any \(B\subset[k]\) with \(\abs{B}\ge 2\) we have \(\abs{m_\circ[B]}\lesssim \rho\eta_\ast^{1-\abs{B}}\) due to~\eqref{f circ expl} where single-block partition of \(B\) yields the worst bound, and thus 
  \begin{equation}\label{mcirc pi}
    \abs*{\prod_{B\in\pi}m_\circ[B]}\lesssim 1+\rho \eta_\ast^{\abs{\pi}-k}
  \end{equation} 
  for any \(\pi\in\NCP[k]\), concluding the proof of~\eqref{norm M bound} using definition~\eqref{eq M1..k def}. Finally, for the proof of~\eqref{tr M bound} note that after taking the trace \(\braket{M_{[k]}A_k}\) with \(M_{[k]}\) as in~\eqref{eq M1..k def} for \(\mathfrak{a}\) traceless \(A_i\) only those \(\pi\in\NCP[k]\) give a non-zero contribution for which \(K(\pi)\) has at most \(\abs{K(\pi)}\le k-\lceil \mathfrak{a}/2\rceil\) blocks. Equivalently, \(\pi\) necessarily has at least \(\abs{\pi}\ge \lceil \mathfrak{a}/2\rceil+1\) blocks, concluding also the proof of~\eqref{tr M bound} using~\eqref{mcirc pi}.  
\end{proof}

\section{Proof of the local law}
\subsection{Alternative representations of \texorpdfstring{\(m,m_\circ,M\)}{m,m0,M}}
To prepare the proof of Theorem~\ref{local law} we first provide explicit alternative representations of the divided differences \(m[\cdot]\) and their free-cumulant version \(m_\circ[\cdot]\) based upon non-crossing graphs instead of non-crossing partitions. We begin with the definition of non-crossing graphs.
\begin{definition}
  Let \(S\subset \N\) be a finite set of integers arranged on a circle as in Definition~\ref{kreweras}. We call an undirected graph \((S,E)\) \emph{crossing} if there exist two edges \((ab),(cd)\in E\) with \(a<c<b<d\), otherwise we call it \emph{non-crossing} and we denote the set of \emph{non-crossing graphs} by \(\NCG(S)\). We call a graph \((S,E)\) a \emph{dissection graph} if for \(S=\set{s_1<\cdots<s_k}\) we have \((s_1s_2),(s_2s_3),\ldots,(s_k s_1)\not\in E\), i.e.\ if all edges dissect the polygon spanned by \(s_1,\ldots, s_k\), and denote the set of dissection graphs by \(\NCG_\mathrm{d}(S)\). Finally, we denote the set of connected graphs by \(\NCG_\mathrm{c}(S)\).
\end{definition}

Each non-crossing graph \((S,E)\in\NCG(S)\) trivially induces a non-crossing partition \(\cP\in\NCP(S)\) with blocks representing the vertices in the connected components of \((S,E)\) and thus we can represent 
\begin{equation}\label{ncg ncp}
  \NCG(S)=\bigsqcup_{\cP\in\NCP(S)}\prod_{B\in\cP}\NCG_\mathrm{c}(B),
\end{equation}
see Figure~\ref{graph part cc} for an example.

\begin{figure}[htbp]
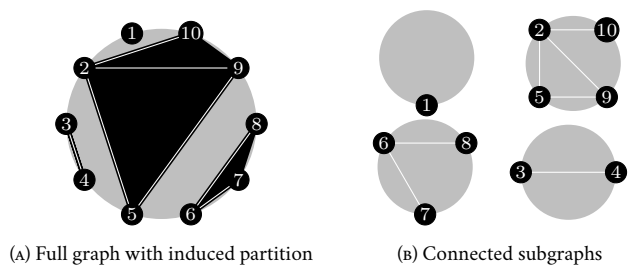

\centering
\subcaptionbox{Full graph with induced partition\label{ex4a}}[.40\linewidth]{
  \circEmb[ptwidth=8pt,number=10,parts={{2,5,9,10},{3,4},{6,7,8}},edges={2/5,5/9,2/9,2/10,3/4,6/8,6/7},vlabel=1,radius=4em]{}}
  \subcaptionbox{Connected subgraphs\label{ex4b}}[.30\linewidth]{
  \circEmb[ptwidth=8pt,vertlabels={1},number=1,vlabel=1,radius=2em]{}\hspace{.8em}%
  \circEmb[ptwidth=8pt,vertlabels={2,5,9,10},edges={2/5,5/9,2/9,2/10},vlabel=1,radius=2em]{}%
  \circEmb[ptwidth=8pt,vertlabels={6,7,8},edges={6/7,8/6},vlabel=1,radius=2em]{}%
  \circEmb[ptwidth=8pt,vertlabels={3,4},edges={3/4},vlabel=1,radius=2em]{}}
\caption{Decomposition of the graph \(\set{(25),(59),(29),(2\,\,10),(34),(68),(67)}\) according to partition \(\set{1|2,5,9,10|3,4|6,7,8}\) of its connected components.}\label{graph part cc}
\end{figure}

\begin{lemma}\label{lemma q}
  Let \(k\in\N\) and let \(z_1,\ldots,z_k\in\C\). For \(S\subset [k]\) and the divided difference \(m[S]:=m[\set{z_s\given s\in S}]\) we have 
  \begin{equation}\label{eq m repr}
    m[S] = \biggl(\prod_{s\in S} m_s\biggr) \sum_{E\in\NCG(S)} q_E, \qquad q_E:=\prod_{ab\in E} q_{ab}, \quad q_{ab}:=\frac{m_a m_b}{1-m_a m_b}
  \end{equation}
  with \(m_a:=m(z_a)\). Moreover, for the free-cumulant function of \(m\) we have 
  \begin{equation}\label{eq mcirc repr}
    m_\circ[S] = \biggl(\prod_{s\in S} m_s\biggr) \sum_{E\in\NCG_\mathrm{c}(S)} q_E.
  \end{equation}
\end{lemma}
We stress that in contrast to the formulas~\eqref{eq m kappa def}--\eqref{f circ expl} for \(f[S]\) and \(f_\circ [S]\) valid for any function \(f\),
the representations~\eqref{eq m repr}--\eqref{eq mcirc repr} of \(m[S]\) and \(m_\circ[S]\) in terms of \(q\) hold only for the specific function \(m\), the solution to equation~\eqref{eq MDE}.
\begin{proof}[Proof of Lemma~\ref{lemma q}]
  We first prove~\eqref{eq m repr} by induction on \(\abs{S}\) with \(\abs{S}=1\) being trivial. The \(\abs{S}=2\) case 
  \begin{equation}\label{eq lemma S2 case}
    m[z_i,z_j] = \frac{m_i-m_j}{z_i-z_j} = \frac{m_i-m_j}{m_j+1/m_j-m_i-1/m_i} = \frac{m_im_j}{1-m_im_j}=q_{ij}=m_im_j(1+q_{ij})
  \end{equation}
  follows directly from~\eqref{eq MDE}.

  For the induction step we may consider, without loss of generality, the set \(S=[n]\) and spectral parameters \(z_1\ne z_n\). The general case follows by relabelling the spectral parameters and continuity in case of equal spectral parameter (for a direct argument in the case of equal spectral parameters see Remark~\ref{lemma q alt proof} below). From the induction hypothesis we may assume that~\eqref{eq m repr} has been established for \(S'=(1,n],[1,n)\). Using the non-crossing property we partition \(\NCG[1,n]\) into 
  \begin{equation}\label{eq G part1} 
    \NCG[1,n] = \NCG^{(1n)}[1,n] \sqcup \NCG^{\lnot(1n)}[1,n],
  \end{equation}
  i.e.\ the subsets of graphs containing or not containing the edge \((1n)\), with 
  \[\NCG^{\lnot(ab)}[a,b]:=\set{E\in\NCG[a,b]\given (ab)\not\in E}, \quad \NCG^{(ab)}[a,b]:=\NCG^{\lnot(ab)}[a,b]\times\set{\set{(ab)}}\] 
  and then further partition \(\NCG^{\lnot(1n)}[n]\) according to the maximal vertex connected to \(1\), i.e. 
  \begin{equation}\label{eq G part2}
    \NCG^{\lnot (1n)}[1,n]= \NCG(1,n] \sqcup \bigsqcup_{j=2}^{n-1}(\NCG^{(1j)}[j]\times \NCG[j,n]).
  \end{equation}
  We then obtain
  \begin{equation}\label{eq graph recursion}
    \begin{split}
      \sum_{E\in \NCG[1,n]} \frac{q_E}{1+q_{1n}} &= \sum_{E\in \NCG^{\lnot(1n)}[1,n]} q_E\\
       &=\sum_{E\in \NCG(1,n]} q_E+\sum_{j=2}^{n-1}m_1m_j\Bigl( \sum_{E\in \NCG[1,j]} q_E \Bigr)\Bigl( \sum_{E\in \NCG[j,n]} q_E \Bigr) \\
      &=  \sum_{E\in \NCG[1,n)} q_E+\sum_{j=2}^{n-1}m_jm_n\Bigl( \sum_{E\in \NCG[1,j]} q_E \Bigr)\Bigl( \sum_{E\in \NCG[j,n]} q_E \Bigr),
    \end{split}
  \end{equation}
  where the first equality in~\eqref{eq graph recursion} follows from~\eqref{eq G part1}, and the second equality then follows from~\eqref{eq G part1}--\eqref{eq G part2} together with \((1+q_{1j})m_1m_j=q_{1j}\). The argument for the third equality is completely symmetric. 
 Then from~\eqref{eq lemma S2 case} and the induction hypothesis we have
  \begin{equation}\label{eq graph rec 2}
    \begin{split}
      \frac{m(1,n]-m[1,n)}{z_n-z_1}  &=\frac{q_{1n}}{m_n-m_1}\Bigl(m(1,n]-m[1,n)\Bigr)\\
      &= \biggl(\prod_{i=1}^n m_i\biggr) \frac{q_{1n}}{m_n-m_1} \biggl(\frac{1}{m_1} \sum_{E\in\NCG(1,n]}q_E - \frac{1}{m_n}\sum_{E\in\NCG[1,n)}q_E  \biggr).
    \end{split}
  \end{equation}
  By solving~\eqref{eq graph recursion} for \(m_1^{-1}\sum_{E\in \NCG(1,n]} q_E\) and \(m_n^{-1}\sum_{E\in \NCG[1,n)} q_E\) 
  and noting that for the difference of the two used in~\eqref{eq graph rec 2} the \(\sum_j\) terms cancel, we obtain 
  \begin{equation*}
    \begin{split} 
      & \frac{m(1,n]-m[1,n)}{z_n-z_1} \\
      & = \biggl(\prod_{i=1}^n m_i\biggr) \frac{q_{1n}}{m_n-m_1} \frac{1}{1+q_{1n}}\biggl(\frac{1}{m_1}-\frac{1}{m_n}\biggr) \sum_{E\in\NCG[1,n]}q_E = \biggl(\prod_{i=1}^n m_i\biggr) \sum_{E\in\NCG[1,n]}q_E,
    \end{split}
  \end{equation*}
  completing the induction step. 

  We now turn to~\eqref{eq mcirc repr} and reorganise~\eqref{eq m repr} in terms of the vertex sets of the connected components as in~\eqref{ncg ncp} and obtain 
  \[ m[S] = \sum_{\cP\in\NCP(S)} \prod_{B\in\cP} \biggl(\Bigl(\prod_{s\in B}m_s\Bigr) \sum_{E\in \NCG_\mathrm{c}(B)} q_E\biggr),\]
  so that~\eqref{eq mcirc repr} follows immediately from Definition~\ref{m kappa def} and the uniqueness of \(m_\circ\), c.f.~\eqref{f circ expl}. 
\end{proof}
\begin{remark}[Alternative proof of Lemma~\ref{lemma q} for equal spectral parameters]\label{lemma q alt proof}
  In the case where all spectral parameters are equal we consider \(S=[n]\) with \(z_1=\cdots=z_n=z\). It is well known that the generating functions \(a_n(w),b_n(w)\) of the graphs in \(\NCG([n]),\NCG_\mathrm{d}([n])\) with \(j\) edges satisfy~\cite[Eq.~(9)]{MR345872}
  \begin{equation}
    \begin{split}
      a_n(w)&=\begin{cases}
        1, & n=1,\\
        1+w, &n=2,\\
        (1+w)^n b_n(w), & n>2.
      \end{cases}\\
    b_{n+1}(w) &=  \begin{cases}
      1, & n=2,\\(1+2w)b_n(w) + \frac{2w(1+w)}{n}b_n'(w),& n\ge 3,
    \end{cases}
    \end{split}
  \end{equation}
  and therefore~\eqref{eq m repr} is equivalent to
  \begin{equation}\label{eq mn an}
    \frac{m^{(n-1)}}{(n-1)!} = m^n a_n(q),
  \end{equation}
  which is obvious for \(n=1\). For \(n=2\) the identity~\eqref{eq mn an} follows from differentiating~\eqref{eq MDE} yielding \((1-m^2)m'=m^2\) and therefore \(m'=q=m^2(1+q)\), and similarly for \(n=3\) we obtain 
  \begin{equation}\label{q prime eq} 
    \frac{m''}{2}=\frac{q'}{2} = \frac{m^3}{(1-m^2)^3}=  m^3 (1+q)^{3} = m^3 a_3(q). 
  \end{equation}
  Assuming~\eqref{eq mn an} for some \(n\ge 3\) and differentiating yields 
  \[ 
    \begin{split}
      \frac{m^{(n)}}{n!} &= \Bigl(m^{n-1} q (1+q)^n+ m^n q'(1+q)^{n-1}\Bigr) b_n(q)   + \frac{m^n q'(1+q)^n b_n'(q)}{n} \\
      &= m^{n+1} (1+q)^{n+1}\Bigl((1+2q)b_n(q)+\frac{2q(1+q)b_n'(q)}{n}\Bigr) = m^{n+1} a_{n+1}(q)
    \end{split}
  \]
  from~\eqref{q prime eq}, completing the induction step. 
\end{remark}
Using Lemma~\ref{lemma q} and~\eqref{ncg ncp} we immediately obtain an alternative representation of \(M_{[k]}\) in the form 
\begin{equation}\label{M 1..k alt}
  \begin{split}
    M_{[k]} &=\Bigl(\prod_{i=1}^k m_i\Bigr) \sum_{\pi\in\NCP[k]} \pTr_{K(\pi)}(A_1,\ldots,A_{k-1}) \prod_{B\in\pi} \biggl(\sum_{E\in\NCG_\mathrm{c}(B)} q_E \biggr)\\
    &= \Bigl(\prod_{i=1}^k m_i\Bigr)\sum_{E\in\NCG[k]} \pTr_{K(\pi(E))}(A_1,\ldots,A_{k-1}) q_E,
  \end{split}
\end{equation}
where \(\pi(E)\in\NCP[k]\) denotes the non-crossing partition induced by the connected components of \(E\in\NCG[k]\). Using~\eqref{M 1..k alt} we obtain a third equivalent (this time recursive) definition of \(M_{[k]}\) and also a simple  tracial recursive relationship expressing the divided difference structure. For \(1\le i<j\le k\) we define \(M_{[i,j]}\) exactly as in~\eqref{eq M1..k def} in terms of the non-crossing partitions \(\pi\in\NCP([i,j])\) and partial traces of the matrices \(A_i,\ldots,A_{j-1}\). For brevity of notations we furthermore set \(M_{(i,j]}:=M_{[i+1,j]}\) and \(M_{[i,j)}:=M_{[i,j-1]}\).
\begin{lemma}
  For any \(k\ge 2\) we have the recursive relations 
  \begin{equation}\label{M1k recursion}
      \begin{split}
        M_{[k]} &=  m_1\biggl( A_1  M_{(1,k]} + q_{1k}\braket{A_1 M_{(1,k]}}+ \sum_{j=2}^{k-1} \braket*{M_{[1,j]}} \Bigl(M_{[j,k]}+q_{1k}\braket*{M_{[j,k]}}\Bigr)\biggr)\\
        &=  m_k\biggl(  M_{[1,k)} A_{k-1}   + q_{1k}\braket{M_{[1,k)}A_{k-1} }+ \sum_{j=2}^{k-1} \braket*{M_{[1,j]}} \Bigl(M_{[j,k]}+q_{1k}\braket*{M_{[j,k]}}\Bigr)\biggr)
      \end{split}
  \end{equation}
  and if \(z_1\ne z_k\) then also
  \begin{equation}\label{eq M1k tr recursion}
    \braket*{M_{[k]}} = \frac{\braket*{M_{[1,k)} A_{k-1} - A_1 M_{(1,k]} }}{z_1-z_k}.
  \end{equation}
\end{lemma}
\begin{proof}
  From the partition~\eqref{eq G part1}--\eqref{eq G part2} and~\eqref{M 1..k alt} we obtain
  \begin{align}\label{M1..k induction proof}
      &\frac{M_{[k]}}{m_1\cdots m_k} \\\nonumber
      &= \sum_{E\in \NCG(1,k]} q_E \Bigl(\pTr_{K(\pi(E))}A_{[1,k)} +q_{1k} \pTr_{K(\pi(E\cup\set{(1k)}))} A_{[1,k)} \Bigr) \\\nonumber
      &\;+ \sum_{j=2}^{k-1} \sum_{E_j\in \NCG^{(1j)}[j]\times \NCG[j,k]} q_{E_j} \Bigl(\pTr_{K(\pi(E_j))}A_{[1,k)}+q_{1k}\pTr_{K(\pi(E_j\cup\set{(1k)}))}A_{[1,k)}\Bigr).
  \end{align}
  Since \(E\in\NCG(1,k]\) implies \(\set{1}\in\pi(E)\) it follows that in \(K(\pi(E))\) both \(1,k\) are in the same block and therefore 
  \[\pTr_{K(\pi(E))}A_{[1,k)}=A_1\pTr_{K(\pi(E))|_{(1,k]}}A_{(1,k)}.\]
  Similarly, for \(E_j=E_j^1\sqcup E_j^2\in \NCG^{(1j)}[j]\times\NCG[j,k]\) we note that 
  \[\pTr_{K(\pi(E_j))}A_{[1,k)}=\pTr_{K(\pi(E_j^1))|_{[j]}}A_{[1,j)} \pTr_{K(\pi(E_j^2))|_{[j,k]}}A_{[j,k)} \] since \(1,j\) are necessarily in the same block from \(\pi(E)\) and therefore \([1,j)\) and \([j,k)\) are in different blocks of \(K(\pi(E))\). Finally, we note that  
  \[\pTr_{K(\pi(E\cup\set{(1k)}))} A_{[1,k)}=\braket{\pTr_{K(\pi(E\cup\set{(1k)}))}A_{[1,k)}}\]
  and 
  \[\sum_{E_j^1\in \NCG^{(1j)}[j]}  q_{E_j^1} \braket{\pTr_{K(\pi(E_j^1))}A_{[1,j)}}=\frac{q_{1j}}{1+q_{1j}}\sum_{E_j^1\in \NCG[j]} q_{E_j^1} \braket{\pTr_{K(\pi(E_j^1))}A_{[1,j)}},\]
  so that~\eqref{M1..k induction proof} yields the first equality in~\eqref{M1k recursion}. The proof of the second equality in~\eqref{M1k recursion} is completely analogous by partitioning the non-crossing graphs according to the edges connected to \(k\) rather than \(1\), hence details are omitted.

  Using both equalities of~\eqref{M1k recursion} we obtain 
  \begin{equation}
    \begin{split}
      \braket*{M_{[k]}} &= \frac{m_1}{m_1-m_k}\braket*{M_{[k]}} - \frac{m_k}{m_1-m_k} \braket*{M_{[k]}}\\
      &= (1+q_{1k})\frac{m_1m_k}{m_1-m_k}\braket*{M_{[1,k)}A_{k-1}-A_1 M_{(1,k]}}
    \end{split}
  \end{equation}
  and~\eqref{eq M1k tr recursion} follows from 
  \[(1+q_{1k})\frac{m_1 m_k}{m_1-m_k}=\frac{q_{1k}}{m_1-m_k}=\frac{1}{z_1-z_k}\] 
  with the last equality due to
  \[ \frac{z_1-z_k}{m_1-m_k} = \frac{m_k-m_1}{m_1-m_k} + \frac{1/m_k-1/m_1}{m_1-m_k} = \frac{1}{m_1m_k}-1 = \frac{1-m_1m_k}{m_1m_k} = \frac{1}{q_{1k}}.\qedhere\]
\end{proof}

\subsection{Proof of Theorem~\ref{local law}}
\label{sec:proofllaw}
The proof of Theorem~\ref{local law} is inductive over \(K\ge 1\). For \(K=k=1\) both~\eqref{iso local law}--\eqref{av local law} follow directly from Theorem~\ref{single G local law}. For the induction we proceed in several steps:
  \begin{enumerate}[label=(S\arabic*)]
    \item\label{av local law step} Proof of~\eqref{av local law} for \(k=K\) assuming~\eqref{av local law} for \(k<K\).
    \begin{enumerate}[label=(S\arabic{enumi}\alph*)]
      \item\label{av local law stepI} Proof of~\eqref{av local law} for \(k=K\), \(A_k=I\) and \(\rho_k=\max_j \rho_j\).
      \item\label{av local law stepZ} Proof of~\eqref{av local law} for \(k=K\) traceless \(\braket{A_k}=0\).
    \end{enumerate}
    \item\label{iso local law step} Proof of~\eqref{iso local law} for \(k=K\) assuming~\eqref{av local law} for \(k\le K\) and~\eqref{iso local law} for \(k<K\).
  \end{enumerate}
  We first note that~\ref{av local law stepI}--\ref{av local law stepZ} imply~\ref{av local law step}. Indeed, by cyclicity we may rearrange \(\braket{G_1A_1\cdots G_k A_k}\) in such a way that \(\rho_k=\max_j\rho_j\) and by decomposing \(A_k=\braket{A_k}I+(A_k-\braket{A_k}I)\) we conclude~\ref{av local law step} from linearity,~\ref{av local law stepI} for \(\braket{A_k}I\) and~\ref{av local law stepZ} for \(A_k-\braket{A_k}I\). It thus remains to establish~\ref{av local law stepI},~\ref{av local law stepZ} and~\ref{iso local law step}. 
  
  The main input for arguments below is the bound on renormalized chains of resolvents established in~\cite[Theorem~4.1, Remark~4.3]{2012.13215} in the form of
  \begin{align}\label{av underline}
    \abs*{\braket{\un{WG_1A_1\cdots G_k A_k}}} \prec \frac{\rho}{N\eta_*^k}\\\label{iso underline}
    \abs*{\braket{\vx,\un{WG_1A_1\cdots G_k }\vy}} \prec \sqrt{\frac{\rho}{N\eta_*}}\frac{1}{\eta_*^{k-1}}
  \end{align}
  Here the renormalization, denoted by underlining, is defined as 
  \[\un{Wf(W)}:=Wf(W)-\wt\E \wt W (\partial_{\wt W} f)(W),\]
  with \(\partial_{\wt W}\) denoting the directional derivative in direction \(\wt W\) and \(\wt W\) denotes an independent GUE-matrix with expectation \(\wt\E\). Using the resolvent identity \(WG-zG=I\), equation~\eqref{eq MDE} and
  \[
  \wt\E\wt W\partial_{\wt W} G=-\wt\E \wt W G \wt W G= -\braket{G}G
  \]
  we obtain 
  \begin{equation}\label{G underline}
    G=m-m\un{WG}+m\braket{G-m}G. 
  \end{equation}
  \begin{proof}[Proof of~\ref{av local law stepI}]
    In case \(\Im z_1 \Im z_k<0\) we use the resolvent identity with the abbreviation \(G_{[a, b]}:=G_a A_a G_{a+1}\cdots A_{b-1}G_b\) to obtain 
  \begin{equation}\label{eq local law sing case}
    \begin{split}
      \braket{G_{[k]}} &= \braket*{\frac{G_1-G_k}{z_1-z_k} A_1 G_{(1,k)} A_{k-1}} \\
      &= \frac{\braket{M_{[1,k)} A_{k-1} - A_1 M_{(1,k]} }}{z_1-z_k}  + \landauOprec*{\frac{1}{N\eta_*^k}} = \braket{M_{[k]}}+ \landauOprec*{\frac{1}{N\eta_*^k}}
    \end{split}
  \end{equation}
  from~\eqref{eq M1k tr recursion}, the averaged local law for \(k-1\) resolvents and \(\abs{z_1-z_k}\ge \eta_1\vee \eta_k\ge \eta_*\). 
  
  In case \(\Im z_1 \Im z_k>0\) we instead use~\eqref{G underline} and
  \[\un{WG_{[k]}}=\un{WG_1}A_1G_{(1,k]} + \sum_{j=2}^{k}\braket{G_{[j]}}G_{[j,k]} \]
  to obtain 
  \begin{equation}\label{eq GGGk resolution}
    \begin{split}
      G_{[k]} & = m_1 A_1 G_{(1,k]} - m_1 \un{W G_{[k]}} + \sum_{j=2}^{k-1} m_1\braket{G_{[j]}} G_{[j,k]} \\
      &\qquad + m_1 m_k \braket{G_{[k]}} + m_1\braket{G_1-m_1}G_{[k]} + m_1\braket{G_{[k]}}(G_k-m_k).
    \end{split}
  \end{equation}
  From the averaged local law~\eqref{av local law} for up to \(k-1\) resolvents we have
  \begin{equation}\label{G1..j local law}
    \braket{G_{[j]}}\braket{G_{[j,k]}}=\braket{M_{[j]}}\braket{M_{[j,k]}}+ \landauOprec*{\frac{\rho}{N\eta_*^k}}
  \end{equation}
  from 
  \[
  \begin{split}
    \abs{\braket{G_{[j]}}}&\le\sqrt{\braket{(G_1A_1)(G_1A_1)^\ast}} \sqrt{\braket{G_{2\cdots j}^\ast G_{2\cdots j}}}\lesssim \frac{\sqrt{\braket{\Im G_1} \braket{\Im G_j}}}{(\eta_1\eta_j)^{1/2}\eta_2\cdots \eta_{j-1}}\le \frac{\rho}{\eta_*^{j-1}},
  \end{split}  \]
  where in the second inequality we used Ward identity \eqref{eq:ward}.
  By taking the averaged trace of~\eqref{eq GGGk resolution} we obtain 
    \begin{equation}\label{eq GGGk resolution2}
      \begin{split}
        &\Bigl(1-m_1m_k-m_1\braket{G_1-m_1}-m_1\braket{G_k-m_k}\Bigr)\braket{G_{[k]} }  \\
        &\qquad = m_1 \braket{A_1 G_{(1,k]}} - m_1 \braket{\un{W G_{[k]}}} + \sum_{j=2}^{k-1} m_1\braket{G_{[j]}} \braket{G_{[j,k]}} \\
        &\qquad = m_1\braket{A_1 M_{(1,k]}} + \sum_{j=2}^{k-1} m_1 \braket{M_{[j]}}\braket{M_{[j,k]}} + \landauOprec*{\frac{\rho}{N\eta_\ast^k}},
      \end{split}
    \end{equation}
    where in the second step we used~\eqref{av underline},~\eqref{G1..j local law} and 
    \[ \braket{A_1 G_{(1,k]}} = \braket{A_1 M_{(1,k]}} + \landauOprec*{\frac{1}{N\eta^{k-1}_\ast}}= \braket{A_1 M_{(1,k]}} + \landauOprec*{\frac{\rho}{N\eta^{k}_\ast}} \]
    from the induction hypothesis. By~\eqref{M1k recursion} the first two terms on the r.h.s.\ of~\eqref{eq GGGk resolution2} can be written as 
    \begin{equation}\label{eq M rec}
      \begin{split}
         m_1\Bigl( \braket{A_1  M_{(1,k]}}+ \sum_{j=2}^{k-1} \braket*{M_{[1,j]}} \braket*{M_{[j,k]}}\Bigr)=\frac{\braket{M_{[k]}}}{1+q_{1k}}=(1-m_1m_k)\braket{M_{[k]}}
      \end{split}
    \end{equation}
    so that due to 
    \[\abs{\braket{G_i-m_i}\braket{M_{[k]}}}\prec \frac{\rho}{N\eta_\ast^k}\] 
    from~\eqref{norm M bound} and Theorem~\ref{single G local law} we finally conclude 
    \begin{equation}\label{eq GGGk resolution3}
      \begin{split}
        &\Bigl(1-m_1m_k-m_1\braket{G_1-m_1}-m_1\braket{G_k-m_k}\Bigr)\braket{G_{[k]}-M_{[k]}} = \landauOprec*{\frac{\rho}{N\eta_\ast^k}}.
      \end{split}
    \end{equation}
    By elementary estimates from the definition of \(m\) in~\eqref{eq MDE} and the fact that \(\Im z_1 \Im z_k>0\) we easily conclude \(\abs{1-m_1m_k}\gtrsim \rho_k=\rho\) and therefore together with \(\abs{\braket{G_i-m_i}}\prec 1/(N\eta_\ast) \ll \rho\) we immediately conclude~\eqref{eq local law sing case}.
  \end{proof} 
  \begin{proof}[Proof of~\ref{av local law stepZ}]
  For such \(A_k\) with \(\braket{A_k}=0\), multiplying~\eqref{eq GGGk resolution} from the rhs.\ and taking the trace, we obtain 
  \begin{equation}
    \begin{split}
      &\Bigl(1+\landauOprec*{\frac{1}{N\eta_*}}\Bigr)\braket{G_{[k]}A_k} \\
      &= m_1 \braket{A_1 G_{(1,k]} A_k} + \sum_{j=2}^{k-1} m_1 \braket{G_{[j]}}\braket{G_{[j,k]} A_k} +\landauOprec*{\frac{1}{N\eta_*^k}}\\
      &= \braket{M_{[k]}A_k} +\landauOprec*{\frac{1}{N\eta_*^k}}
    \end{split}  
  \end{equation}
  from~\eqref{M1k recursion} and the local laws for up to \(k-1\) resolvents. 
\end{proof}
  \begin{proof}[Proof of~\ref{iso local law step}]
  We take the inner product of~\eqref{eq GGGk resolution} with \(\vx,\vy\) to obtain 
  \begin{equation}\label{iso 1st step}
    \begin{split}
      &\braket{\vx,G_{[k]}\vy}\\
      &= m_1\braket{\vx,A_1G_{(1,k]}\vy} + \sum_{j=2}^{k-1} m_1\braket{G_{[j]}} \braket{\vx,G_{[j,k]}\vy} + m_1m_k\braket{\vx,\vy}\braket{G_{[k]}} + \landauOprec*{\sqrt{\frac{\rho}{N\eta_*}} \frac{1}{\eta_*^{k-1}} }\\
      &= m_1\braket{\vx,A_1 M_{(1,k]}\vy} + \sum_{j=2}^{k-1} m_1\braket{M_{[j]}}\braket{\vx,M_{[j,k]}\vy} + m_1m_k \braket{\vx,\vy}\braket{M_{[k]}} + \landauOprec*{\sqrt{\frac{\rho}{N\eta_*}} \frac{1}{\eta_*^{k-1}} } 
    \end{split}
  \end{equation}
  from the renormalization bound~\eqref{iso underline}, the averaged local laws~\eqref{av local law} for up to \(k\) resolvents and isotropic local laws~\eqref{iso local law} up to \(k-1\) resolvents. Now from the recursive relation~\eqref{M1k recursion} and \((1+q_{1k})m_1m_k=q_{1k}\) we have
  \begin{align}
      &m_1 A_1 M_{(1,k]} + \sum_{j=2}^{k-1} m_1\braket{M_{[j]}} M_{[j,k]} + m_1m_k \braket{M_{[k]}}= M_{[k]}
  \end{align}
  and the claim follows together with~\eqref{iso 1st step}.
\end{proof}
By combining~\ref{av local law step}--\ref{iso local law step} we conclude the induction step and thereby the proof of Theorem~\ref{local law}.

\printbibliography%

\end{document}